\documentclass[11pt]{article}
\usepackage[a4paper, margin=1in]{geometry}
\usepackage{amsmath, amssymb, amsthm}
\usepackage{mathtools}
\usepackage{graphicx}
\usepackage{paralist}
\usepackage{bbm}
\usepackage{float}
\usepackage{enumerate}
\usepackage[colorlinks = true,
            linkcolor = blue,
            urlcolor  = blue,
            citecolor = blue,
            anchorcolor = blue]{hyperref}
\usepackage{kbordermatrix}
\usepackage{url, hyperref}
\usepackage{cleveref}
\usepackage{csquotes}
\usepackage{tikz}

\theoremstyle{plain}
\newtheorem{theorem}{Theorem}[section]
\newtheorem*{theorem*}{Theorem}
\newtheorem{proposition}[theorem]{Proposition}
\newtheorem{lemma}[theorem]{Lemma}
\newtheorem{corollary}[theorem]{Corollary}

\theoremstyle{definition}
\newtheorem{definition}[theorem]{Definition}
\newtheorem{notation}[theorem]{Notation}
\newtheorem{remark}[theorem]{Remark}
\newtheorem{example}[theorem]{Example}

\newcommand{\R}{\mathbb{R}}
\newcommand{\Z}{\mathbb{Z}}

\newcommand{\RC}[1]{\mathcal{R}_\vec{0}(#1)}
\newcommand{\Rh}{\mathcal{R}_h}

\newcommand{\RS}[1]{\mathcal{R}(#1)}

\renewcommand{\NG}[1]{\mathrm{NG}(#1)}
\newcommand{\PGL}[1]{\mathrm{PGL}(\R,#1)}

\renewcommand{\vec}[1]{\mathbf{#1}}
\renewcommand{\v}{\vec{v}}
\renewcommand{\u}{\vec{u}}
\renewcommand{\a}{\vec{a}}
\renewcommand{\b}{\vec{b}}

\newcommand{\w}{\vec{w}}
\newcommand{\x}{\vec{x}}
\newcommand{\y}{\vec{y}}
\newcommand{\twentyfourcell}{C_4^{(24)}}

\DeclareMathOperator{\conv}{conv}

\DeclareMathOperator{\rank}{rank}

\DeclareMathOperator{\pyr}{pyr}
\DeclareMathOperator{\bipyr}{bipyr}

\DeclarePairedDelimiter\ceil{\lceil}{\rceil}
\DeclarePairedDelimiter\floor{\lfloor}{\rfloor}

\newcommand{\wh}{\widehat}

\title{On the Dimensions of the Realization Spaces of Polytopes}
\begin{small}
\author{Laith Rastanawi\thanks{Supported by the DFG Research Training Group GRK 2434 “Facets of Complexity”}\\
Institut f\"ur Mathematik\\Freie Universität Berlin\\Arnimallee 2\\14195 Berlin, Germany\\
\url{laith.rastanawi@fu-berlin.de}
\and
Rainer Sinn\\
Institut f\"ur Mathematik\\Freie Universität Berlin\\Arnimallee 2\\14195 Berlin, Germany\\
\url{rainer.sinn@fu-berlin.de }
\and
G\"unter M.~Ziegler\setcounter{footnote}{6}\thanks{Research supported by the DFG Collaborative Research Center TRR~109 ``Discretization in Geometry and Dynamics''}\\
Institut f\"ur Mathematik\\Freie Universität Berlin\\Arnimallee 2\\14195 Berlin, Germany\\
\url{ziegler@math.fu-berlin.de}}
\end{small}

\date{July 1, 2020}

\begin{document}
\maketitle

\begin{abstract} 
Robertson (1988) suggested a model for the realization space of a convex $d$-dimensional polytope
and an approach via the implicit function theorem, which -- in the case of a full rank Jacobian -- proves that the realization space is a manifold of dimension $\NG{P}:=d(f_0+f_{d-1})-f_{0,d-1}$, which is the natural guess for the dimension given by the number of variables minus the number of quadratic equations that are used in the definition of the realization space.
While this indeed holds for many natural classes of polytopes (including simple and simplicial polytopes, as well as 
all polytopes of dimension at most~$3$),
and Robertson claimed this to be true for all polytopes,
Mn\"ev's (1986/1988) Universality Theorem implies that it is not true in general:
Indeed, (1) the centered realization space is not a smoothly embedded manifold in general, and (2) it does not have the dimension $\NG{P}$ in general.

In this paper we develop Jacobian criteria for the analysis of realization spaces. From these we get easily that for various large and natural classes of polytopes the realization spaces are indeed manifolds, whose dimensions are given by $\NG{P}$. 
However, we also identify the smallest polytopes where the dimension count (2) and thus Robertson's claim fails,
among them the bipyramid over a triangular prism.
For the property (1), we analyze the classical 24-cell:
We show that the realization space has at least the dimension $\NG{\twentyfourcell}=48$,
and it has points where it is a manifold of this dimension, but it is not smoothly embedded as a manifold everywhere.
\end{abstract}

\section{Introduction} \label{sec:introduction}
The study of geometric realizations of convex polytopes goes back to Legendre in 1794, who asked the following question \cite[p.~309]{Legendre1794}:
\emph{How many variables are needed to determine a geometric realization of a given (combinatorial type of a) polytope?}
In modern terms, Legendre asks for the dimension of the space of geometric realizations of a polytope $P$, 
i.e.~the space of all choices of coordinates for the vertices of~$P$ that lead to a polytope with the same (isomorphic) face lattice.

The case of polygons is straightforward: The number of parameters is two times the number of vertices,
which we would now write as $2f_0=2f_1$. The first major step is due to Legendre himself \cite[Note~VIII]{Legendre1794} and Steinitz \cite[Sec.~69]{Steinitz1976}, who settled the question in dimension $3$, where the number of variables turns out to be the number of edges plus $6$.
So the Legendre--Steinitz Theorem says that the realization space of any $3$-dimensional polytope $P$ is a manifold of dimension $f_1(P)+6$.

It is natural to ask Legendre's question for $d$-dimensional polytopes ($d$-polytopes, for short).
An answer was given by Robertson \cite[Theorem, p.~18]{robertson1984} in 1984: 
\emph{The realization space of any $d$-polytope is a smooth submanifold of $\R^{d(f_0+f_{d-1})}$ 
of dimension $d(f_0+f_{d-1})-\mu$},
where $f_0=f_0(P)$ is the number of vertices and $f_{d-1}=f_{d-1}(P)$ is the number of facets of~$P$,
and $\mu=\mu(P)=f_{0,d-1}(P)$ is the number of vertex-facet incidences.
For the proof of his claim, Robertson represented the realization space as an open subset (defined by strict quadratic inequalities)
of the solution set of $\mu$ quadratic equations in $\R^{d(f_0+f_{d-1})}$.
In such a setting, it is natural to expect that the solution set has the dimension
given by the number of variables minus the number of equations,
and Robertson's proof is then “built round a simple application of the implicit function theorem” \cite[p.~18]{robertson1984}.

With the Euler--Poincaré equation it is not hard to check that Robertson's claim agrees with what we know in dimensions $d=1,2,3$.
For simplicial and for simple polytopes, the realization spaces can be seen as open subspaces of $\R^{df_0}$ resp.\ $\R^{df_{d-1}}$ (see below).
However, Mn\"ev's Universality Theorem for polytopes from 1986/1988
(Mn\"ev \cite{Mnev1988}, see also Richter-Gebert \cite{Jurgen1996}, as well as the exposition in \cite{Z45}
and Mn\"ev's web page \url{http://www.pdmi.ras.ru/~mnev/bhu.html})
implies that the realization space is not in general a manifold: For any semi-algebraic set defined over the integers
one can construct a polytope whose realization space modulo affine transformations is equivalent to $M$ up to
certain trivial fibrations. This implies that realization spaces can have very complicated topology (locally, as well as globally),
so realization spaces are not manifolds in general,
but it does not have immediate implications on the dimension of the realization space.

So it appears that Robertson's claimed theorem is true for the polytopes that we see occurring “in nature”,
but it is false for very special examples that arise by complicated constructions
in the proofs of Mn\"ev, Richter-Gebert, et al.

In this paper we start from Robertson's work.
The model for the realization space suggested by him, formalized as the \emph{centered realization space} by Adiprasito \& Ziegler \cite{Z130,ZA2011}
(see Definition~\ref{def:RC}) is indeed very natural and convenient,
as it is \emph{by definition} a semi-algebraic set (indeed, an open subset of a real-algebraic variety), so topology and metric are clear
and the dimension is well-defined, but also the implicit function theorem is directly applicable in the way that Robertson set it up.
From this we get 
\begin{compactitem}
	\item a natural and rather general sufficient criterion for the validity of Robertson's claim (see Section~\ref{sec:jacobian_and_degeneracy}),
    \item a very simple and natural proof (see {\cite[Lemma 2.8]{BDH2008}}) for the Legendre--Steinitz theorem for $3$-dimensional polytopes (Corollary \ref{cor:3polytopes}), as well as
	\item a natural tool for the analysis of other classes, with positive as well as negative results.
\end{compactitem}

Our next step is the search for counter-examples (Section~\ref{sec:4dim}), where we identify the unique three smallest counter-examples to Robertson's claim, for which the realization space is still a smooth manifold, but its dimension is \emph{not} given by the “natural guess” of “number of variables minus number of equations.” One of these is very simple: The bipyramid over a triangular prism (Section~\ref{subsec:three_smallest_counter-examples}).

Finally, we go for an iconic object in polytope theory, the 24-cell.
We know from Paffenholz' thesis \cite{Paffenholz-diss} that the 24-cell is not projectively unique, but its realization space (in our model) has dimension at least 28, which is the dimension of the group of projective transformations on~$\R^4$ plus~$4$.
We construct new classes of realizations of the 24-cell, and from these we derive in Section~\ref{sec:4dim} that
\begin{compactitem}
	\item the realization space (in our model) has dimension at least $48$, which is the dimension of the group
  of projective transformations on $\R^4$ plus $24$, and indeed this is the “natural guess” dimension
  predicted by Robertson's claim;
	\item indeed, there are points in the realization space where locally the realization space is a manifold of dimension~$48$
	(Corollary \ref{cor:smoothmanifold24}),
	\item but there are also points in the realization space (such as those given by
	realizations as a regular polytope)
	where the realization space is \emph{not} a smooth submanifold of the ambient space $\R^{d(f_0+f_{d-1})}=\R^{192}$
	(Theorem~\ref{thm:c24notsmooth}).
\end{compactitem}
We doubt that the realization space is a topological manifold, and 
indeed we are not sure that it is pure (has the same local dimension everywhere), but this is left open: The 24-cell keeps some of its mystery.

\section{Realization Spaces of Polytopes} \label{sec:background}

For general facts about polytopes, we refer to \cite{Ziegler1995}.
Let $P \subset \R^d$ be a $d$-dimensional polytope (or $d$-polytope, for short).
We write $f_i(P)$ (or simply $f_i$, if the polytope is clear from the context) for the total number of faces of~$P$ that have dimension $i$.
In particular, $f_0(P)$ is the number of vertices and $f_{d-1}(P)$ is the number of facets.
We write $f_{0,d-1}(P)$ for the number of vertex-facet incidences.
We call a polytope $P\subset\R^d$ \emph{centered} if it contains the origin $\vec{0}$ in its interior.
We can represent every polytope as the convex hull $P = \conv(V)$ of its vertices,
where $V$ is a $(d \times f_0)$-matrix and $\conv(V)$ is the convex hull of the columns of $V$.
By rescaling the facet defining inequalities, we can represent every centered polytope as the intersection of half-spaces
\[
	P = \{ \x \in \R^d \mid A^t \vec{x} \leq \vec{1} \},
\]
where $A$ is a $(d \times f_{d-1})$-matrix.
In these two representations, the matrices $V$ and $A$ are unique up to column permutations.
From now on, we label the vertices and facets to make these matrices unique.
In this case, we say that $P$ is \emph{labeled}, and we call $(V,A)$ the \emph{combined vertex and facet description} of~$P$.

A labeled $d$-polytope $Q \subset \R^d$ is said to \emph{realize} $P$ if there exists an isomorphism between the face lattices of~$P$ and $Q$
that respects the labeling of their vertices and facets.
If $Q$ was centered, we say that $Q$ is a \emph{centered realization} of~$P$.

We now define our model for the realization space of a polytope.

\begin{definition}[Centered realization space {\cite{Z130,ZA2011}}] \label{def:RC}
    Let $P$ be a labeled $d$-polytope with $n$ vertices and $m$ facets.
    The \emph{centered realization space} of~$P$ is the set
    \[
	\RC{P} := \left\{ (V,A) \in \R^{d \times (n+m)} \mid \conv(V) = \{\x \in \R^d \mid A^t\x \leq \vec{1}\} \text{ realizes $P$} \right\}.
	\]
    That is, $\RC{P}$ is the set of combined vertex and facet descriptions of centered realizations of~$P$.
\end{definition}

The next proposition shows that the centered realization space has a nice description as a (basic) semi-algebraic set (over $\Z$),
that is, the solution set of a system of polynomial equations and inequalities. In particular, its dimension is well-defined.

\begin{proposition} \label{prop:semialgebraic_desc}
    Let $P$ be a $d$-polytope with vertices $\v_1, \cdots, \v_n$ and facets $F_1, \cdots, F_m$.
    The centered realization space of~$P$ is equal to the set
    \begin{align*}
        \RC{P} = \left\{
        (W,B) =
        (\w_1, \cdots, \w_n, \b_1, \cdots, \b_m) \in \R^{d \times (n+m)}
		\quad\middle|\quad
		 \w_i^t \b_j
		 \Big\{
        \begin{array}{ll}
            =1 & \textrm{ if } \v_i \in F_j    \\[-1mm]
            <1 & \textrm{ if } \v_i \not\in F_j
        \end{array}
        \right\}.
    \end{align*}
\end{proposition}

\begin{proof}
    The equality $\conv(W) = \{ \x \in \R^d \mid B^t \x \leq \vec{1}\}$ holds since the hyperplanes
    $\{ \x \in \R^d \mid \x^t \b_j = 1\}$ are facet-defining hyperplanes for $\conv(W)$.
    The polytope $\conv(W)$ realizes $P$
    because it has the same vertex-facet incidence structure as $P$
    (this determines the face lattice of the polytope, see \cite[Exercise 2.7]{Ziegler1995})
    with the correct labels.
\end{proof}

Our model of the realization space behaves nicely with respect to duality.
Recall that the \emph{polar polytope} $P^\Delta$ of a labeled centered $d$-polytope $P$ is the polytope
\[
	P^\Delta := \{ \y \in \R^d \mid \y^t \x  \leq 1 \textrm{ for all } \x \in P \}.
\]
There is an inclusion-reversing bijection between the face lattices of~$P$ and $P^\Delta$.
This bijection maps the vertices (resp.\ the facets) of~$P$ onto the facets (resp.\ the vertices) of $P^\Delta$.
We assume that the vertices (resp.\ the facets) of $P^\Delta$ are labeled with the labels of the facets (resp.\ the vertices) of~$P$ induced from the bijection.

\begin{proposition}
    For every labeled centered $d$-polytope $P$ we have $\RC{P} \cong \RC{P^\Delta}$,
	where the isomorphism is given by the permutation $(V,A)\mapsto (A,V)$.
    In particular, $\dim \RC{P} = \dim \RC{P^\Delta}$.
\end{proposition}

\section{The Jacobian and the Degeneracy Criteria} \label{sec:jacobian_and_degeneracy}

According to Proposition~\ref{prop:semialgebraic_desc}, the centered realization space $\RC{P}$
is a semi-algebraic set defined by quadratic equations and strict quadratic inequalities,
so it may be seen as an open subset (cut out by the strict inequalities)
of a fiber of the characteristic map defined by the equations.
This interpretation of Robertson's work on \cite[p.~19]{robertson1984} yields the setup for applying the implicit function theorem in this context.

\begin{definition}[The characteristic map] \label{def:charmap}
    Let $P$ be a $d$-polytope with vertices $\v_1, \cdots \v_n$ and facets $F_1, \cdots, F_m$.
    Let $\mu := f_{0,d-1}(P)$ denote the number of vertex-facet incidences of~$P$.
    The \emph{characteristic map} of~$P$, denoted $\Phi_P$, is the map
    \begin{align*}
        \Phi_P: \R^{d \times (n+m)}     & \rightarrow \R^\mu \\
        (W, B)  & \mapsto \left( \w_i^t \b_j - 1 \right)_{\substack{[i,j], \\ \v_i \in F_j}}.
    \end{align*}
    We order the entries $\w_i^t \b_j - 1$ lexicographically into the vector $\Phi_P(W, B)$, i.e.~$\w_i^t \b_j - 1$ occupies the $[i,j]$th entry.
\end{definition}

Clearly, this map sends (the vertex and facet description of) any centered realization of~$P$ to $\vec{0} \in \R^\mu$.
Thus $\RC{P}$ is an open subset of the fiber $\Phi_P^{-1}(\vec{0})$.
The heuristic that the solution set of a system of equations has dimension
“number of variables minus number of equations”
suggests that $\dim \RC{P}$ is the number of variables $d f_0(P)+ d f_{d-1}(P)$ minus the number of equations $f_{0,d-1}(P)$.
This is what we call the “natural guess” for $\dim \RC{P}$ (formerly called the “naive guess” in \cite{Z130,ZA2011}).

\begin{definition}[The “natural guess” for the dimension of the realization space] \label{def:ng}
    Let $P$ be a $d$-polytope. The \emph{natural guess} for 
	$\dim \RC{P}$, the dimension of the centered realization space of~$P$, is
    \[
	\NG{P} := d \big(f_0(P)+f_{d-1}(P)\big) - f_{0,d-1}(P).
	\]
\end{definition}

From an algebraic point of view, the equations $\w_i^t\b_j = 1$ do not necessarily generate a nice ideal. It is certainly not prime in general, as we argue next.
\begin{remark}
    For any $d$-polytope $P$ with $n$ vertices and $m$ facets,
    the fiber $\Phi_P^{-1}(\vec{0})$ of the characteristic map contains
	certain subsets of ``degenerate configurations.'' Namely, there is
    \begin{itemize}
        \item the set $S_V(P)$ where all columns of $V$ are identical (corresponding to identical points)
        and the $m$ affine hyperplanes contain that point, which has dimension
	    \[
		d + m(d-1) = (m+1)d-m = (m+1)(d-1) + 1,
		\]
        \item and the set $S_A(P)$ where all columns of $A$ are identical (corresponding to identical affine hyperplanes)
        and the $n$ points lie on that affine hyperplane, which has dimension
    	\[
		d + n(d-1) = (n+1)d-n = (n+1)(d-1) + 1.
		\]
    \end{itemize}
    These subsets are not contained in $\RC{P}$ because all of the strict inequalities are violated. Yet the equations defining $\RC{P}$ still hold.
\end{remark}

\begin{example}
    The $3$-cube $C_3 \subset \R^3$ has $8$ vertices, $6$ facets and $24$ vertex-facet incidences.
    We will later give a proof for the Legendre--Steinitz Theorem, which in this special case
	gives that $\dim(\RC{C_3}) = \NG{C_3} = 18$.
	However, the set of ``degenerate configurations'' $S_A(C_3)$ for the cube has dimension~$19$.
\end{example}

\subsection{Examples}
In the following proposition, we calculate the natural guess for special classes and common constructions of polytopes.

\begin{proposition}\label{prop:ng}
    Let $P$ be a $d$-polytope.
    \begin{enumerate}[\rm(i)]
        \item If $P$ is simplicial, then $\NG{P} = df_0(P)$.
        \item If $P$ is simple, then $\NG{P} = df_{d-1}(P)$.
        \item If $P$ is a $3$-polytope, then $\NG{P} = f_{1}(P) + 6$.
        \item If $P = \pyr(Q)$ is a pyramid over a $(d-1)$-polytope $Q$, then $\NG{P} = \NG{Q} + 2d$.
        \item If $P = \bipyr(Q)$ is a bipyramid over a $(d-1)$-polytope $Q$, then $\NG{P} = 2\NG{Q} + (2-d)f_0(Q) + 2d$.
    \end{enumerate}
\end{proposition}

\begin{proof}
    (i) Each facet has exactly $d$ vertices. Therefore, $f_{0,d-1}(P) = df_{d-1}$.\\
    (ii) This is the dual statement of (i).\\
    (iii) Each edge of a $P$ determines $4$ vertex-facet incidences, while each vertex-facet incidence corresponds to $2$ edges.
    Therefore, $2f_{0,2}(P) = 4f_1(P)$. Using Euler's formula for $3$-polytopes, we get
    \[
	\NG{P} = 3\left( f_0(P) + f_2(P) \right) - f_{0,2}(P) = f_1(P) + 6.
	\]
    (iv) Here we have $f_0(P) = f_0(Q) + 1$ and $f_{d-1}(P) = f_{d-2}(Q) + 1$.
    Every vertex of $Q$ lies in the facet $Q$ of~$P$ and in a facet of~$P$ that is the pyramid over a facet $F$ of $Q$ if and only if it lies in~$F$.
    This gives us $f_0(Q) + f_{0,d-2}(Q)$ incidences.
    Additionally, the new vertex of the pyramid (the apex) lies in $f_{d-2}(Q)$ many facets of~$P$.
    Thus, $f_{0,d-1}(P) = f_0(Q) + f_{0,d-2}(Q) + f_{d-2}(Q)$.
    Therefore, the natural guess is
    \begin{align*}
        \NG{P} &= d\left( f_0(Q) + 1 + f_{d-2}(Q) + 1 \right) - f_0(Q) - f_{0,d-2}(Q) - f_{d-2}(Q) \\
               &= (d-1)\left( f_0(Q) + f_{d-2}(Q) \right) - f_{0,d-2}(Q) + 2d = \NG{Q} + 2d.
    \end{align*}
    (v) Here we have $f_0(P) = f_0(Q) + 2$ and $f_{d-1}(P) = 2f_{d-2}(Q)$.
    Each vertex of $Q$ lies in two facets of~$P$ that is the two pyramids over a facet $F$ of $Q$ if and only if it lies in~$F$.
    This gives us $2f_{0,d-2}(Q)$ incidences.
    Finally, the new two vertices of the bipyramid (the apexes) each lie in $f_{d-2}(Q)$ facets of~$P$.
    Thus, $f_{0,d-1}(P) = 2f_{0,d-2}(Q) + 2f_{d-2}(Q)$.
    Therefore, the natural guess is
    \begin{align*}
        \NG{P} &= d\left( f_0(Q) + 2 + 2f_{d-2}(Q) \right) - 2f_{0,d-2}(Q) - 2f_{d-2}(Q) \\
               &= 2(d-1) \left( f_0(Q) + f_{d-2}(Q) \right) - 2f_{0,d-2}(Q) +(2-d)f_0(Q) + 2d \\
               &= 2 \NG{Q} + (2-d)f_0(Q) + 2d.
               \qedhere
    \end{align*}
\end{proof}

As we will see later, the natural guess is equal to the dimension of the realization space for large, very common classes of polytopes.
However, this is not true in general. The following examples show that the natural guess can be negative.

\begin{example}[Adiprasito \& Ziegler \cite{ZA2011}]
    A $d$-polytope is called \emph{cubical} if its facets are combinatorially isomorphic to the standard $(d-1)$-cube $[-1,1]^{d-1}$.
    A \emph{neighborly cubical polytope} is a cubical $d$-polytope which has the ($\floor{\frac{d}{2}} - 1$)-skeleton of a cube.
    For any $n \geq d \geq 2r+2$, Joswig and Ziegler in \cite{JZ2000} constructed a cubical $d$-polytope $\mathrm{NCP}_d(n)$ whose $r$-skeleton is combinatorially equivalent to that of the $n$-cube.
    The number of vertices of $\mathrm{NCP}_d(n)$ is $2^n$, and the number of its facets is given by
    \[
	f_{d-1}(\mathrm{NCP}_d(n)) = 2d + 4 \sum_{p=0}^{n-d-1}\left( \binom{\floor{\frac{d}{2}} + p + 1}{p + 2} + \binom{\floor{\frac{d+1}{2}} + p}{p + 2} \right) 2^p,
	\]
    see \cite[Corollary 18]{JZ2000}.
    Each facet of $\mathrm{NCP}_d(n)$ has $2^{d-1}$ vertices,
    and thus $f_{0,d-1}(\mathrm{NCP}_d(n)) = 2^{d-1}f_{d-1}(\mathrm{NCP}_d(n))$.
    Now the natural guess of $\mathrm{NCP}_d(n)$ is a function of $d$ and $n$ and we can compute it.
    \begin{itemize}
        \item If $d = 4$, then $f_{3}(\mathrm{NCP}_4(n)) = (n-2)2^{n-2}$, and thus $\NG{\mathrm{NCP}_4(n)} = (6-n)2^n < 0$ whenever $n \geq 7$.
        \item If $d \geq 5$, then $\NG{\mathrm{NCP}_d(n)} < 0$ for $n \geq d+1$.
    \end{itemize}
\end{example}

\begin{example}\label{examples:bipyramid-over-dual-to-cyclic}
    Let $Q:=C_{d-1}(n)^\Delta$ be the polar of a $(d-1)$-dimensional (centered) cyclic polytope with $n$ vertices.
    The number of facets of $Q$ is $n$, and the number of vertices of $Q$ is given by
    \begin{align*}
        f_0(Q) = \binom{n - \ceil{\frac{d-1}{2}}}{\floor{\frac{d-1}{2}}} + \binom{n - \floor{\frac{d-1}{2}} -1}{\ceil{\frac{d-1}{2}}-1},
    \end{align*}
    see \cite[4.7.3]{Grunbaum2003}.
    Since $Q$ is simple, we have $\NG{Q} = (d-1)f_{d-2}(Q)$.
    Finally, let $P$ be a bipyramid over $Q$.
    The natural guess of~$P$ is then given by
    \begin{align*}
        \NG{P} &= 2\NG{Q} + (2-d)f_0(Q) + 2d \\
               &= 2(d-1)f_{d-2}(Q) + (2-d)f_0(Q) + 2d.
    \end{align*}
    We can easily compute $\NG{P}$ as a function of $n$ and $d$.
    For instance, if $d=5$, then $\NG{P} < 0$ for $n \geq 10$, and if $d=6$, then $\NG{P} < 0$ for $n \geq 9$.
\end{example}

We note that in these examples, $\RC{P}$ may still be a manifold,
(and for the bipyramids over duals of cyclic polytopes this is indeed the case, see Proposition \ref{prop:simplenecessarilyflat}), but certainly its dimension is not $\NG{P}$.

\subsection{The Jacobian Criterion}

Now we are ready to state our first main theorem, which is an application of the implicit function theorem.

\begin{theorem}[The Jacobian Criterion for $d$-polytopes; cf.~Robertson {\cite[p.~19]{robertson1984}}] \label{thm:jacobian}
    Let $P$ be a $d$-polytope.
    If the Jacobian matrix $J_{\Phi_P}(V_0, A_0)$ of the characteristic map $\Phi_P$ at some point $(V_0, A_0) \in \RC{P}$ has full row rank 
	(that is, if it has rank $f_{0,d-1}(P)$),
    then $\RC{P}$ is, in a neighborhood of $(V_0, A_0)$, a smooth manifold of dimension $\NG{P}$.
    In particular, $\dim_{(V_0, A_0)} \RC{P} = \NG{P}$ and $\dim \RC{P} \geq \NG{P} \ge 0$.
\end{theorem}

\begin{proof}
    By the implicit function theorem (see for example \cite[Theorem 5.15]{Lee2001}),
    $\Phi_P^{-1}(\vec{0})$ is, in a neighborhood of $(V_0, A_0)$, a smooth manifold of dimension $\NG{P}$.
    The result follows since $\RC{P}$ is an open subset of $\Phi_P^{-1}(\vec{0})$.
\end{proof}

\begin{remark}
    More generally, even if the Jacobian matrix does not have full rank, its corank still gives an upper bound for the local dimension:
    \[
        \dim_{(V_0, A_0)} \RC{P} \leq d(n+m) - \rank J_{\Phi_P}(V_0, A_0).
    \]
    Indeed, if the rank is $r$ we select $r$ rows of the Jacobian 
	that form a submatrix of full rank $r$, and then apply the Implicit Function Theorem
	to the corresponding map $\Phi_P^{(r)}: \R^{d \times (n+m)}     \rightarrow \R^r$.
	obtained by restricting the characteristic map to the corresponding $r$ components.	
\end{remark}

The previous theorem gives us the motivation to study the structure of the Jacobian matrix of $\Phi_P$ for a polytope $P$.

\begin{notation}
    Let $P$ be a $d$-polytope with $n$ vertices, $m$ facets and $\mu$ vertex-facet incidences.
    The \emph{Jacobian matrix in compressed notation} of $\Phi_P$ at $(V_0, A_0) \in \R^{d \times (n+m)}$,
    denoted by $J_{\Phi_P}^{c}(V_0, A_0)$,
    is the matrix we get from the matrix $J_{\Phi_P}(V_0, A_0)$ by replacing the columns indexed by
    $v_{i_1}, \cdots, v_{i_d}$ (resp.\ $a_{j_1}, \cdots, a_{j_d}$)
    by one column indexed by $\v_i$ (resp.\ $\a_j$) whose entries are the corresponding row vectors.
    Note that
    \begin{center}
        $J_{\Phi_P}^{c}(V_0, A_0)$ is a ($\mu \times (n+m)$)-matrix whose entries are row vectors,
    \end{center}
    while
    \begin{center}
        $J_{\Phi_P}(V_0, A_0)$ is a ($\mu \times d(n+m)$)-matrix.
    \end{center}
\end{notation}

\begin{lemma}
    $J_{\Phi_P}^{c}(V, A)$ has $\vec{a}_j^t$ in the $[i,j]$th row and the $i$th column,
    and it has $\vec{v}_i^t$ in the $[i,j]$th row and the $(n+j)$th column.
    All other entries are zero.
\end{lemma}

\begin{proof}
    This follows since, for each $k \in \{1, \cdots, d\}$ we have
    \[
    \frac{\partial (\v_i^t \a_j - 1)}{\partial v_{i,k}} = a_{j,k}
    \qquad \textrm{ and } \qquad
    \frac{\partial (\v_i^t \a_j - 1)}{\partial a_{i,k}} = v_{i,k}. \qedhere
    \]
\end{proof}

\begin{example}
    Consider the $3$-polytope $P$ shown below,
	whose facet labels are indicated below in the image of a projection onto the facet
    $F_1 = \{\v_1, \v_2, \v_3, \v_4\}$.

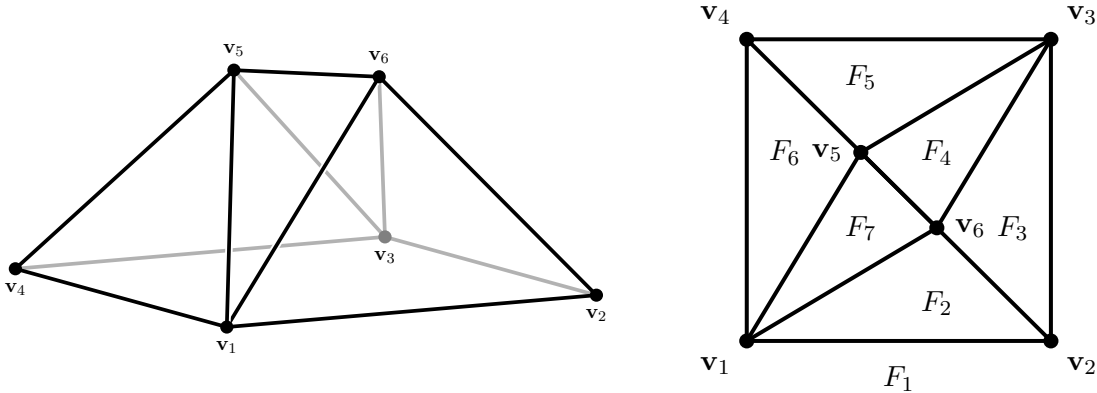
    \begin{figure}[H]
		\begin{minipage}{.5\textwidth}
            \centering
            \scalebox{.7}{
\begin{tikzpicture}
    [x  = {(0.496534251504211cm,-0.138003984746687cm)},
     y  = {(0.868011745659684cm,0.0754650600195354cm)},
     z  = {(0.00305720132416302cm,0.987552492230283cm)},
     scale = 2]

    \tikzstyle{vertex} = [fill=black, circle, inner sep=2.5pt]
    \tikzstyle{bvertex} = [fill=gray, circle, inner sep=2.5pt]
    \tikzstyle{edge} = [draw=black, line width=2pt]
    \tikzstyle{dedge} = [draw=white, line width=4pt]
    \tikzstyle{bedge} = [draw=gray, line width=2pt, opacity=0.6]
    \tikzstyle{face} = [draw=none, fill=none, fill opacity=0.35]

    \coordinate (v1) at (2, -2, 0);
    \coordinate (v2) at (2, 2, 0);
    \coordinate (v3) at (-2, 2, 0);
    \coordinate (v4) at (-2, -2, 0);
    \coordinate (v6) at (0.5, 0.5, 2);
    \coordinate (v5) at (-0.5, -0.5, 2);

    \draw[bedge] (v3) -- (v2);
    \draw[bedge] (v3) -- (v6);
    \draw[bedge] (v3) -- (v5);
    \draw[bedge] (v3) -- (v4);

    \node at (v3) [bvertex, label=below:$\mathbf{v}_3$] {};

    \draw[face] (v1) -- (v2) -- (v6) -- (v5) -- (v4) -- cycle;

    \draw[edge] (v1) -- (v2);
    \draw[dedge] (v1) -- (v6);
    \draw[edge] (v1) -- (v6);
    \draw[dedge] (v1) -- (v5);
    \draw[edge] (v1) -- (v5);
    \draw[edge] (v1) -- (v4);
    \draw[edge] (v2) -- (v6);
    \draw[edge] (v5) -- (v4);
    \draw[edge] (v5) -- (v6);

    \node at (v1) [vertex, label=below:$\mathbf{v}_1$] {};
    \node at (v2) [vertex, label=below:$\mathbf{v}_2$] {};
    \node at (v4) [vertex, label=below:$\mathbf{v}_4$] {};
    \node at (v5) [vertex, label=above:$\mathbf{v}_5$] {};
    \node at (v6) [vertex, label=above:$\mathbf{v}_6$] {};
\end{tikzpicture}
    }
        \end{minipage}%
		\begin{minipage}{.5\textwidth}
		    \centering
\begin{tikzpicture}
    [scale=1]

    \tikzstyle{vertex} = [fill=black, circle, inner sep=2pt]
    \tikzstyle{edge} = [draw=black, line width=1.5pt]

    \coordinate (v1) at (-2, -2);
    \coordinate (v2) at (2, -2);
    \coordinate (v3) at (2, 2);
    \coordinate (v4) at (-2, 2);
    \coordinate (v5) at (-1/2, 1/2);
    \coordinate (v6) at (1/2, -1/2);

	\draw[edge] (v1) -- (v2) -- (v3) -- (v4) -- cycle;
    \draw[edge] (v5) -- (v6);
    \draw[edge] (v1) -- (v5) -- (v3);
    \draw[edge] (v1) -- (v6) -- (v3);
    \draw[edge] (v2) -- (v5) -- (v4);

    \coordinate[label=center:$F_2$] (a2) at (1/2,-3/2);
    \coordinate[label=center:$F_3$] (a3) at (3/2,-1/2);
    \coordinate[label=center:$F_4$] (a4) at (1/2,1/2);
    \coordinate[label=center:$F_5$] (a5) at (-1/2,3/2);
    \coordinate[label=center:$F_6$] (a6) at (-3/2,1/2);
    \coordinate[label=center:$F_7$] (a7) at (-1/2,-1/2);
    \coordinate[label=center:$F_1$] (a1) at (0,-5/2);

    \node at (v1) [vertex, label=below left:$\mathbf{v}_1$] {};
    \node at (v2) [vertex, label=below right:$\mathbf{v}_2$] {};
    \node at (v3) [vertex, label=above right:$\mathbf{v}_3$] {};
    \node at (v4) [vertex, label=above left:$\mathbf{v}_4$] {};
    \node at (v5) [vertex, label=left:$\mathbf{v}_5$] {};
    \node at (v6) [vertex, label=right:$\mathbf{v}_6$] {};
\end{tikzpicture}
		\end{minipage}
	\end{figure}

    The matrix $J_{\Phi_P}^{c}(V, A)$ at a point $(V, A) \in \R^{3 \times (6+7)}$ is

    \renewcommand{\arraystretch}{1.3}
	\begin{align*}
	\kbordermatrix{
			  & \v_1    & \v_2    & \v_3    & \v_4    & \v_5    & \v_6    & \a_1    & \a_2    & \a_3    & \a_4    & \a_5    & \a_6    & \a_7    \\
    \	[1,1] & \a^t_1  & \vec{0} & \vec{0} & \vec{0} & \vec{0} & \vec{0} & \v^t_1  & \vec{0} & \vec{0} & \vec{0} & \vec{0} & \vec{0} & \vec{0} \\
    \	[1,2] & \a^t_2  & \vec{0} & \vec{0} & \vec{0} & \vec{0} & \vec{0} & \vec{0} & \v^t_1  & \vec{0} & \vec{0} & \vec{0} & \vec{0} & \vec{0} \\
	\	[1,6] & \a^t_6  & \vec{0} & \vec{0} & \vec{0} & \vec{0} & \vec{0} & \vec{0} & \vec{0} & \vec{0} & \vec{0} & \vec{0} & \v^t_1  & \vec{0} \\
	\	[1,7] & \a^t_7  & \vec{0} & \vec{0} & \vec{0} & \vec{0} & \vec{0} & \vec{0} & \vec{0} & \vec{0} & \vec{0} & \vec{0} & \vec{0} & \v^t_1  \\
	\	[2,1] & \vec{0} & \a^t_1  & \vec{0} & \vec{0} & \vec{0} & \vec{0} & \v^t_2  & \vec{0} & \vec{0} & \vec{0} & \vec{0} & \vec{0} & \vec{0} \\
	\	[2,2] & \vec{0} & \a^t_2  & \vec{0} & \vec{0} & \vec{0} & \vec{0} & \vec{0} & \v^t_2  & \vec{0} & \vec{0} & \vec{0} & \vec{0} & \vec{0} \\
	\	[2,3] & \vec{0} & \a^t_3  & \vec{0} & \vec{0} & \vec{0} & \vec{0} & \vec{0} & \vec{0} & \v^t_2  & \vec{0} & \vec{0} & \vec{0} & \vec{0} \\
    \	[3,1] & \vec{0} & \vec{0} & \a^t_1  & \vec{0} & \vec{0} & \vec{0} & \v^t_3  & \vec{0} & \vec{0} & \vec{0} & \vec{0} & \vec{0} & \vec{0} \\
	\	[3,3] & \vec{0} & \vec{0} & \a^t_3  & \vec{0} & \vec{0} & \vec{0} & \vec{0} & \vec{0} & \v^t_3  & \vec{0} & \vec{0} & \vec{0} & \vec{0} \\
	\	[3,4] & \vec{0} & \vec{0} & \a^t_4  & \vec{0} & \vec{0} & \vec{0} & \vec{0} & \vec{0} & \vec{0} & \v^t_3  & \vec{0} & \vec{0} & \vec{0} \\
	\	[3,5] & \vec{0} & \vec{0} & \a^t_5  & \vec{0} & \vec{0} & \vec{0} & \vec{0} & \vec{0} & \vec{0} & \vec{0} & \v^t_3  & \vec{0} & \vec{0} \\
	\	[4,1] & \vec{0} & \vec{0} & \vec{0} & \a^t_1  & \vec{0} & \vec{0} & \v^t_4  & \vec{0} & \vec{0} & \vec{0} & \vec{0} & \vec{0} & \vec{0} \\
	\	[4,5] & \vec{0} & \vec{0} & \vec{0} & \a^t_5  & \vec{0} & \vec{0} & \vec{0} & \vec{0} & \vec{0} & \vec{0} & \v^t_4  & \vec{0} & \vec{0} \\
	\	[4,6] & \vec{0} & \vec{0} & \vec{0} & \a^t_6  & \vec{0} & \vec{0} & \vec{0} & \vec{0} & \vec{0} & \vec{0} & \vec{0} & \v^t_4  & \vec{0} \\
	\	[5,4] & \vec{0} & \vec{0} & \vec{0} & \vec{0} & \a^t_4  & \vec{0} & \vec{0} & \vec{0} & \vec{0} & \v^t_5  & \vec{0} & \vec{0} & \vec{0} \\
	\	[5,5] & \vec{0} & \vec{0} & \vec{0} & \vec{0} & \a^t_5  & \vec{0} & \vec{0} & \vec{0} & \vec{0} & \vec{0} & \v^t_5  & \vec{0} & \vec{0} \\
	\	[5,6] & \vec{0} & \vec{0} & \vec{0} & \vec{0} & \a^t_6  & \vec{0} & \vec{0} & \vec{0} & \vec{0} & \vec{0} & \vec{0} & \v^t_5  & \vec{0} \\
	\	[5,7] & \vec{0} & \vec{0} & \vec{0} & \vec{0} & \a^t_7  & \vec{0} & \vec{0} & \vec{0} & \vec{0} & \vec{0} & \vec{0} & \vec{0} & \v^t_5  \\
	\	[6,2] & \vec{0} & \vec{0} & \vec{0} & \vec{0} & \vec{0} & \a^t_2  & \vec{0} & \v^t_6  & \vec{0} & \vec{0} & \vec{0} & \vec{0} & \vec{0} \\
	\	[6,3] & \vec{0} & \vec{0} & \vec{0} & \vec{0} & \vec{0} & \a^t_3  & \vec{0} & \vec{0} & \v^t_6  & \vec{0} & \vec{0} & \vec{0} & \vec{0} \\
	\	[6,4] & \vec{0} & \vec{0} & \vec{0} & \vec{0} & \vec{0} & \a^t_4  & \vec{0} & \vec{0} & \vec{0} & \v^t_6  & \vec{0} & \vec{0} & \vec{0} \\
	\	[6,7] & \vec{0} & \vec{0} & \vec{0} & \vec{0} & \vec{0} & \a^t_7  & \vec{0} & \vec{0} & \vec{0} & \vec{0} & \vec{0} & \vec{0} & \v^t_6  \\
	}.
	\end{align*}
    \renewcommand{\arraystretch}{1}

    Notice that the columns are indexed by the vertices and the facets, while the rows are indexed by the vertex-facet incidences.
\end{example}

\begin{definition}[The vertex-facet incidence graph]
    The \emph{vertex-facet incidence graph} of a polytope $P$ is the undirected graph
    $\Gamma_P = (\mathcal{V} \cup \mathcal{F}, E)$,
    where $\mathcal{V}$ is the set of all vertices of~$P$,
    $\mathcal{F}$ is the set of all facets of~$P$, and
    \[
	E = \left\{ \{\v, F\} \subset \mathcal{V}\cup \mathcal{F} \mid \v \in F \right\}.
	\]
\end{definition}

Since there are no edges among the nodes in $\mathcal{V}$, nor among the nodes in $\mathcal{F}$, the graph $\Gamma_P$ is bipartite.

\begin{definition}[$k$-degenerate graphs]
    An undirected graph $G$ is \emph{$k$-degenerate} if there exists an ordering of its nodes
    in which each node has at most $k$ neighbors appearing after it in this ordering.
\end{definition}

The degeneracy of a graph was defined by Lick \& White in \cite{LW1970},
where they established some basic properties of degenerate graphs.
One of them is the following proposition.

\begin{proposition}[Lick \& White {\cite{LW1970}}]
    Let $G = (V,E)$ be a $k$-degenerate graph with $|V| \geq k$. Then
    \[
	|E| \leq k|V| - \binom{k+1}{2}.
	\]
\end{proposition}

\begin{proof}
    We prove this statement by induction on $|V|$.
    If $|V| = k$, then $|E|$ can be at most $\binom{k}{2}$.
    On the other side of the inequality, we have
    \begin{align*}
        k|V| - \binom{k+1}{2} &= k^2 - \frac{k(k+1)}{2} = \frac{k(k-1)}{2} = \binom{k}{2}.
    \end{align*}
    Now let $G$ be a $k$-degenerate graph with $|V| \geq k$ and fix an order of the nodes in which each node has at most $k$ neighbors appearing after it in this ordering. In particular, the first node $x_1$ has degree at most $k$.
    Let $G' = (V', E')$ be the graph obtained from $G$ by deleting $x_1$.
    Then $G'$ is still $k$-degenerate and it has $|V| - 1$ nodes and $|E| - \deg(x_1)$ edges.
    Thus, we have
    \begin{align*}
        |E| &= |E'| + \deg(x_1) \leq k(|V| - 1) - \binom{k+1}{2} + \deg(x_1) \leq k|V| - \binom{k+1}{2}.
        \qedhere
    \end{align*}
\end{proof}

Applying the previous proposition on the vertex-facet incidence graph of
a polytope, we immediately get the following corollary.

\begin{corollary}
    Let $P$ be a $d$-polytope and assume that its vertex-facet incidence graph $\Gamma_P$ is $k$-degenerate. Then the inequality
    \[
	f_{0,d-1}(P) \leq k \left( f_0(P)+f_{d-1}(P) \right) - \binom{k+1}{2}
	\]
    holds. In particular, if $\Gamma_P$ is $d$-degenerate, then $\NG{P} \geq \binom{d+1}{2}$. \qed
\end{corollary}

The following theorem is inspired by the proof for Steinitz's Theorem given by Borisov, Dickinson \& Hastings \cite[Lemma 2.8]{BDH2008}.

\begin{theorem}[The Degeneracy Criterion] \label{thm:degeneracy}
    Let $P$ be a $d$-polytope.
    If the vertex-facet incidence graph $\Gamma_P$ has a $d$-degenerate ordering of its nodes such that the following two conditions hold.
    \begin{enumerate}[\rm(i)]
        \item For each facet in this ordering,
        the vertices appearing after it and connected to it by an edge are linearly independent in any centered realization of~$P$.
        \item For each vertex in this ordering,
        the normal vectors of the facets appearing after it and connected to it by an edge are linearly independent in any centered realization of~$P$.
    \end{enumerate}
    Then $\RC{P}$ is a smooth manifold of dimension $\NG{P}$.
\end{theorem}

\begin{proof}
    Let $\v_1, \cdots, \v_n$ and $F_1, \cdots, F_m$ denote the vertices and the facets of~$P$ respectively.
    We will show that the Jacobian Criterion (Theorem \ref{thm:jacobian}) is satisfied at each point of $\RC{P}$ by block triangularizing the Jacobian matrix.
    Let $J^{c} := J^{c}_{\Phi_P}(V_0, A_0)$ denote the Jacobian matrix in compressed notation
    of the characteristic map $\Phi_P$ of~$P$ at some centered realization $(V_0,A_0) \in \RC{P}$.
    First reorder the columns of $J^{c}$ in the way given by the degeneracy.
    Then, process these columns one by one from left to right.
    If the column is indexed by a vertex $\v_i$, move the rows indexed by 
    \[
	\{ [i,j] \mid \v_i \in F_j, F_j \textrm{ appears after } \v_i \textrm{ in the ordering} \}
    \]
    to the bottom of $J^{c}$. If the column is indexed by a normal vector $\a_j$, move the rows indexed by 
    \[
	\{ [i,j] \mid \v_i \in F_j, \v_i \textrm{ appears after } F_j \textrm{ in the ordering} \}
    \]
    to the bottom of $J^{c}$.
    After doing this for all the columns, we get an upper block-triangularized matrix.
    Now $J^{c}$ has full rank if all of these blocks have full row rank, which is guaranteed by conditions (i) and (ii).
    The statement follows using the Jacobian Criterion at all centered realizations.
\end{proof}

\begin{remark} \label{remark:degeneracy}
    Let $G$ be the graph obtained from $\Gamma_P$ by deleting some $r$ edges such that $G$ is $d$-degenerate and satisfies the conditions (i) and (ii) from the previous theorem. Then 
    $$\dim \RC{P} \leq \NG{P} + r.$$
    This corresponds to finding a $((\mu-r) \times d(n+m))$-submatrix of the Jacobian matrix that has full rank at all centered realizations of~$P$.
\end{remark}

The Degeneracy Criterion is not purely combinatorial.
The conditions (i) and (ii) might be satisfied for some geometric centered realizations and fail for others.
However, in the next section, we will derive some purely combinatorial results from the Degeneracy Criterion.

Before moving on to applications, we set up a (scaled) homogeneous version of the results of this section, which turns out to be useful below. Let $C\subset \R^{d+1}$ be a closed and pointed polyhedral cone of dimension $d+1$. Analogously to the centered realization space, we define a primal-dual realization space model for $C$ as follows
\begin{align*}
    \Rh(C) = \left\{ (W,B)\in \R^{(d+1)\times (n+m)} \mid {\rm cone}(W) = \{x\in\R^{d+1}\vert \, B^tx \leq 0\} \text{ realizes }C, \right. \\
    \left. ||\w_i||^2 = ||\b_j||^2 = 1 \right\},
\end{align*}
where $n$ is the number of extreme rays and $m$ the number of facets of $C$.
Analogously to the centered realization space, this set can be described as a semi-algebraic set. Let $\v_1, \cdots, \v_n$ and $F_1, \cdots, F_m$ denote the extreme rays and the facets of $C$ respectively. Then $\Rh(C)$ is equal to the set
    \begin{align*}
        \Rh(C) = \left\{
        (W,B) \in \R^{d+1 \times (n+m)}
		\quad\middle|\quad
		 \w_i^t \b_j
		 \Big\{
        \begin{array}{ll}
            =0 & \textrm{ if } \v_i \in F_j    \\[-1mm]
            <0 & \textrm{ if } \v_i \not\in F_j
        \end{array}
        ,
        \sum_{k=1}^{d+1}{w_{ik}^2}=\sum_{k=1}^{d+1}{b_{jk}^2}=1
        \right\}.
    \end{align*}

\begin{theorem}[The Homogeneous Degeneracy Criterion] \label{thm:degeneracyhom}
    Let $C$ be a closed and pointed polyhedral cone of dimension $d+1$. Assume that its ray-facet incidence graph $\Gamma_C$ has a  $d$-degenerate ordering of its nodes such that the following conditions hold. 
    \begin{enumerate}[\rm(i)]
        \item For each facet in this ordering, the ray generators appearing after it and connected to it by an edge are linearly independent in any realization of $C$.
        \item For every ray in this ordering, the normal vectors of the facets appearing after it and connected to it by an edge are linearly independent in any realization of $C$.
    \end{enumerate}
    Then the realization space $\Rh(C)$ is a smooth manifold of dimension $d(n+m) - \mu$, where $n$ is the number of extreme rays of $C$, $m$ is the number of facets of $C$ and $\mu$ is the number of ray-facet incidences.
\end{theorem}

\begin{proof}
    Let $\v_1, \ldots, \v_n$ be the extreme rays and $F_1,\ldots,F_m$ the facets of $C$.
    As before, let $J^c$ denote the Jacobian matrix in compressed notation of the set of equations defining $\Rh(C)$ at the point $(W, B) \in \Rh(C)$.
    We show that this matrix has full rank, namely equal to $\mu + (n+m)$.
    Each of the last $n+m$ rows corresponds to an equation of the form $||\w_i||^2=1$ or of the form $||\b_j||^2=1$.
    In $J^c$, the only non-zero entry in a row corresponding to the equation $||\w_i||^2 = 1$ is $2\w_i$ and it is in the column indexed by $\v_i$.
    Similarly, the only non-zero entry in a row corresponding to the equation $||\b_j||^2 = 1$ is $2\b_j$ and it is in the column indexed by $F_j$.

    First, we reorder the columns of $J^{c}$ by the degeneracy order of the incidence graph.
    Then, we process these columns one by one from left to right.
    If the column is indexed by a vertex $\v_i$, move the rows indexed by 
    \[
        \{ [i,j] \mid \v_i \in F_j, F_j \textrm{ appears after } \v_i \textrm{ in the ordering} \} \cup \{ \text{The row indexed by } ||\v_i||^2 = 1 \} 
    \]
    to the bottom of $J^{c}$. If the column is indexed by a normal vector $\a_j$, move the rows indexed by 
    \[
        \{ [i,j] \mid \v_i \in F_j, \v_i \textrm{ appears after } F_j \textrm{ in the ordering} \} \cup \{ \text{The row indexed by } ||\a_j||^2 = 1\}
    \]
    to the bottom of $J^{c}$.
    After doing this for all the columns, we get an upper block-triangularized matrix.
    Now $J^{c}$ has full rank if all of these blocks have full row rank.
    Each block has at most $d+1$ vectors. By the conditions (i) and (ii), all of them but the one coming from $\|w_i\|^2=1$ or $\|\b_j\|^2 = 1$ are linearly independent. This last one is orthogonal to the other vectors in the block by the incidence relation of vertices and facets.
    Thus, each block has full rank.
    The statement follows using the implicit function theorem at all points of $\Rh(C)$.
\end{proof}

\section{Applications of the Degeneracy Criterion}\label{sec:apps}

\begin{definition}[Vertex set and facet set of a face]
    Let $P$ be a $d$-polytope and let $F$ be a face of~$P$.
    The \emph{vertex set} of $F$ is the set of all vertices of~$P$ contained in $F$.
    The \emph{facet set} of $F$ is the set of all facets of~$P$ containing $F$.
\end{definition}

The following proposition is easy.
However, combining it with the Degeneracy Criterion produces interesting combinatorial results.
\begin{proposition} \label{prop:rank_prop}
    Let $P \subset \R^d$ be a polytope containing the origin in its interior.
    \begin{enumerate}[\rm(i)]
        \item Let $S$ be a set of $k \leq d$ vertices that lie on a facet of~$P$.
        If some $k-1$ vertices from $S$ are the vertex set of a $(k-2)$-face,
        then the vertices in $S$ are linearly independent.
        \item Let $S$ be a set of $k \leq d$ facets that share a vertex of~$P$.
        If some $k-1$ facets from $S$ are the facet set of a $(d-k+1)$-face,
        then the normal vectors of the facets in $S$ are linearly independent.
    \end{enumerate}
\end{proposition}

\begin{proof}
    The second statement (ii) is the dual of (i), so we will only prove (i).
    Let $F$ be the facet of~$P$ which the vertices of $S$ lie on.
    Let $\v_1, \cdots, \v_{k-1} \in S$ be the vertices of~$P$ which form a $(k-2)$-face,
    and $\v_k$ be the last vertex in $S$.
    Since $\v_1, \cdots, \v_k$ form a $(k-2)$-face (that is, a $(k-2)$-simplex), they are affinely independent.
    By the definition of a proper face, there is an affine hyperplane in $\R^d$ which contains all these vertices and does not contain $\v_k$.
    The affine hull of $\v_1, \cdots, \v_{k-1}$ is contained in this hyperplane,
    and thus the affine hull of $\v_1, \cdots, \v_{k-1}$ does not contain $\v_k$.
    Thus, $\v_1, \cdots, \v_k$ are affinely independent.
    Since the hyperplane spanned by $F$ does not contain $\vec{0}$, these vertices are also linearly independent.
\end{proof}

If the size of $S$ is at most $3$, some of the assumptions in the previous proposition can be dropped.

\begin{proposition} \label{prop:three_rank_prop}
    Let $P \subset \R^d$ be a polytope containing the origin in its interior.
    \begin{enumerate}[\rm(i)]
        \item Let $S$ be a set of $k \leq 3$ vertices that lie on a facet of~$P$.
        Then the vertices in $S$ are linearly independent.
        \item Let $S$ be a set of $k \leq 3$ facets that share a vertex of~$P$.
        Then the normal vectors of the facets in $S$ are linearly independent.
    \end{enumerate}
\end{proposition}

\begin{proof}
    This is true since any three vertices on a facet are affinely independent and the affine span of any facet cannot contain the origin in a centered realization. Again, (ii) is the dual statement of (i).
\end{proof}

\subsection{Almost 3-degenerate Polytopes}\label{sec:combresults}

\begin{theorem} \label{thm:almost_3degen}
    Let $P$ be a $d$-polytope.
    Let $\Pi$ be the graph obtained from the vertex-facet incidence graph $\Gamma_P$ by removing the nodes of degree $d$.
    If $\Pi$ is $3$-degenerate, then $\RC{P}$ is a smooth manifold of dimension $\NG{P}$.
\end{theorem}

\begin{proof}
    We will use the Degeneracy Criterion (Theorem \ref{thm:degeneracy}) to prove this statement.
    The following ordering of the nodes of $\Gamma_P$ is $d$-degenerate.
    First put all the nodes of $\Gamma_P$ of degree $d$. These correspond to the simple vertices and the simplex facets.
    Then put the nodes of $\Pi$ ordered by a $3$-degenerate ordering.
    By Proposition \ref{prop:rank_prop},
    the conditions (i) and (ii) of the Degeneracy Criterion are satisfied at the nodes of $\Gamma_P$ of degree $d$.
    By Proposition \ref{prop:three_rank_prop},
    the conditions (i) and (ii) of the Degeneracy Criterion are satisfied at the remaining nodes of $\Gamma_P$, which are exactly those nodes that are connected to at most $3$ later nodes in the $d$-degenerate ordering we constructed.
\end{proof}

This is our main tool that we apply to special classes of polytopes to show that their realization spaces are manifolds. Before we begin with these applications, we again record a homogeneous version for later use.
\begin{theorem}\label{thm:almost3degenhom}
    Let $C$ be a closed and pointed polyhedral cone of dimension $d+1$. Let $\Pi$ be the graph obtained from its ray-facet incidence graph $\Gamma_C$ by removing the nodes of degree $d$. If $\Pi$ is $3$-degenerate, then $\Rh(C)$ is a smooth manifold of dimension $d(m+n) - \mu$.
\end{theorem}

The proof is the same as that of \Cref{thm:almost_3degen} with the only exception that we use use \Cref{thm:degeneracyhom} instead of the non-homogeneous version \Cref{thm:degeneracy}.

Here is a strict simplification of \Cref{thm:almost_3degen}.
\begin{corollary}
    Let $P$ be a $d$-polytope.
    If 
    \begin{enumerate}[\rm(i)]
        \item each vertex of~$P$ lies in at most $3$ non-simplex facets, or
        \item each facet of~$P$ contains at most $3$ non-simple vertices,
    \end{enumerate}
    then $\RC{P}$ is a smooth manifold of dimension $\NG{P}$. \qed
\end{corollary}

This applies to simple and simplicial polytopes, of course. The corresponding natural guess is computed in \Cref{prop:ng}.
\begin{corollary}[Simple and Simplicial Polytope] \label{cor:simple_and_simplicial}
    Let $P$ be a $d$-polytope. 
    \begin{enumerate}[\rm(i)]
        \item If $P$ is simplicial, then $\RC{P}$ is a smooth manifold of dimension $df_0(P)$.
        \item If $P$ is simple, then $\RC{P}$ is a smooth manifold of dimension $df_{d-1}(P)$. \qed
    \end{enumerate}
\end{corollary}

Polygons are always simple and simplicial, so we get that their realization spaces are always manifolds.
\begin{corollary}[$2$-polytopes] \label{cor:2polytopes}
    Let $P$ be a $2$-polytope, then $\RC{P}$ is a smooth manifold of dimension $2f_0(P) = 2f_1(P)$. \qed
\end{corollary}

To get the analogue for $3$-polytopes, we use the following combinatorial observation for planar graphs. Planar for us means that we can draw the graph in the plane without edges crossing but we do not insist on the edges being line segments.
\begin{proposition} \label{prop:bipartite_planar}
    Bipartite planar graphs are $3$-degenerate.
\end{proposition}

\begin{proof}
    Let $G$ be a bipartite planar graph with $v$ nodes and $e$ edges.
    We first show that $e \leq 2v - 4$. Consider a planar drawing of $G$ in the plane.
    Since the graph is bipartite, every connected component of the complement of $G$ (called a face of the drawing) is bounded by at least $4$ edges, while every edge is in $2$ faces.
    Thus, $4f \leq 2e$, where $f$ is the number of faces of the drawing of $G$.
    Using Euler's formula for planar graphs, we get $e=v+f-2 \leq v + \frac{1}{2}e-2$. Thus, $e \leq 2v-4$.

    Now we get a $3$-degenerate order of $G$ recursively by deletion because our face count shows that a bipartite planar graph always contains a vertex of degree at most $3$.
\end{proof}

\begin{corollary}[$3$-polytopes, see Borisov, Dickinson \& Hastings {\cite[Lemma 2.8]{BDH2008}}] \label{cor:3polytopes}
    Let $P$ be a $3$-polytope. Then $\RC{P}$ is a smooth manifold of dimension $\NG{P}=f_1(P) + 6$.
\end{corollary}

\begin{proof}
    By Proposition \ref{prop:ng} (iii) $\NG{P} = f_1(P) + 6$.
	The vertex-facet incidence graph $\Gamma_P$ of~$P$ is planar,
    since we can draw it in the plane without crossing edges as follows.
    Draw one additional point on each facet of~$P$, and connect it by edges to the vertices of that facet.
    Then project the vertices of~$P$, the new points and the new edges from an interior point of~$P$ onto a sphere that contains $P$.
    Now apply a stereographic projection to the plane to get a planar graph.
    By Proposition \ref{prop:bipartite_planar}, $\Gamma_P$ is $3$-degenerate.
    We are done by Theorem \ref{thm:almost_3degen}.
\end{proof}

\subsection{Hypersimplices} \label{sec:hypersimplices}

The hypersimplices are key examples in polytope theory, which first
appeared in the work of Gabri\'elov, Gel'fand \& Losik \cite{GGL} on
characteristic classes. 

\begin{definition}[The hypersimplex; see {\cite{deLoeraRambauSantos2010} or \cite{Z99}}]
    The \emph{standard hypersimplex} $\Delta_d(k)$ is the polytope defined by
	\[
	\Delta_d(k) := \conv \bigg\{ \v \in \{0,1\}^d \mid \sum_{i=1}^{d} v_i = k \bigg\},
	\]
	where $1 \leq k \leq d-1$.
\end{definition}

Note that $\Delta_d(k) \subset \R^d$ is a $(d-1)$-polytope
since it lies in the affine hyperplane defined by $\sum_{i=1}^d x_i = k$.
It has $\binom{d}{k}$ vertices (the number of ways to choose exactly $k$ ones in a zero-one vector in $\R^d$).
Note also that $\Delta_d(k)$ is affinely isomorphic to $\Delta_d(d-k)$ under the map
$\vec{x} \mapsto \vec{1} - \vec{x}$.
For $k = 1$ or $k = d-1$, we get a $(d-1)$-simplex $\Delta_{d-1}$.
Thus, we are particularly interested in the cases when $k$ lies between $2$ and $d-2$.

We first collect information about the facets of hypersimplices. See 
Ziegler \cite[Sect.~3]{Z99}, De Loera, Rambau \& Santos \cite[Sect.~6.3.6]{deLoeraRambauSantos2010} and Paffenholz \& Ziegler \cite[Sect.~3.3.1]{Z89} 
for more information.

\begin{proposition}\label{prop:hypersimplex_structure}
	For $2 \leq k \leq d-2$, the following statements hold.
	\begin{enumerate}[\rm(i)]
		\item $\Delta_d(k)$ has $2d$ facets:
            $d$ of them are combinatorially isomorphic to $\Delta_{d-1}(k)$,
            and the other $d$ facets are combinatorially isomorphic to $\Delta_{d-1}(k-1)$.
		\item Each vertex of $\Delta_d(k)$ lies on $d$ facets:
            $d-k$ of them are combinatorially isomorphic to $\Delta_{d-1}(k)$,
            and the other $k$ facets are combinatorially isomorphic to $\Delta_{d-1}(k-1)$.
        \item Each $d-3$ facets of $\Delta_d(2)$ of the form $\Delta_{d-1}(2)$ are the facet set of a triangular $2$-face.
	\end{enumerate}
\end{proposition}

 \begin{proof}
	 All three parts are easy to prove. Parts (i) and (ii) are also quite well-known, see \cite[Proposition~6.3.15]{deLoeraRambauSantos2010} and its proof.
	 
     (i) Facets are defined by two types of hyperplanes,
     \begin{align*}
         &H_i^{(0)} = \{ \x \in \R^d \mid x_i = 0 \} \text{ and } \\
         &H_i^{(1)} = \{ \x \in \R^d \mid x_i = 1 \},
     \end{align*}
     for $1 \leq i \leq d$.
     The first type $H_i^{(0)}$ produces facets of the form $\Delta_{d-1}(k)$,
     while the second type $H_i^{(1)}$ produces facets of the form $\Delta_{d-1}(k-1)$.

     (ii) Each vertex has $d-k$ zeros, and thus it lies on $d-k$ facets of the form $\Delta_{d-1}(k)$.
     It also has $k$ ones, and thus it lies on $k$ facets of the form $\Delta_{d-1}(k-1)$.

     (iii) Fix $d-3$ facets of the form $\Delta_{d-1}(2)$.
     Without loss of generality, assume that these $d-3$ facets are defined by the first $d-3$ hyperplanes
     \[
 	H_i^{(0)} = \{ \x \in \R^d \mid x_i = 0 \}, \text{ for } 1 \leq i \leq d-3.
 	\]
     Thus, the face of $\Delta_d(2)$ which they define has the following vertex set.
     \begin{align*}
         &\{\v \in \{0,1\}^d \mid v_1 = \cdots = v_{d-3} = 0, \; v_{d-2} + v_{d-1} + v_{d} = 2 \} \\
         = &\{(\vec{0}, 1, 1, 0), (\vec{0}, 1, 0, 1), (\vec{0}, 0, 1, 1) \} \subset \R^d,
     \end{align*}
     which is a triangular $2$-face.
     This face clearly does not lie on any of the following hyperplanes.
     \[
 	H_i^{(0)} = \{ \x \in \R^d \mid x_i = 0 \}, \text{ for } d-2 \leq i \leq d,
 	\]
     and
     \begin{align*}
         H_i^{(1)} &= \{ \x \in \R^d \mid x_i = 1 \}, \text{ for } 1 \leq i \leq d.
         \qedhere
     \end{align*}
 \end{proof}

\begin{theorem}
    The realization space of $\Delta_d(2)$ is a smooth manifold of dimension
	$\NG{\Delta_d(2)}=\frac{3}{2}(d^2 - d)$.
\end{theorem}

Grande, Padrol, Sanyal proved in \cite{grande2018} that the realization space of hypersimplices $\Delta_d(2)$ is topologically a ball of dimension $\frac32 (d^2-d)$.

\begin{proof}
    We will use the Degeneracy Criterion (Theorem \ref{thm:degeneracy}) to prove this statement.
    Each vertex of $\Delta_d(2)$ lies in $d$ facets, and we have $\binom{d}{2}$ vertices. Thus,
    \[
	f_{0,d-2}(\Delta_d(2)) = d\binom{d}{2}.
	\]
    In particular,
    \[
	\NG{\Delta_d(2)} = (d-1)\left( \binom{d}{2} + 2d \right) - d\binom{d}{2} = \frac{3}{2}(d^2 - d).
	\]

    Let $P \subset \R^{d-1}$ be a centered realization of $\Delta_d(2)$.
    The following ordering of the nodes of the vertex-facet incidence graph $\Gamma_P$ is $(d-1)$-degenerate.
    First put all the facets of the form $\Delta_{d-1}(1)$ at the beginning. These are $(d-2)$-simplices.
    Then put all the vertices, and at the end put all the facets of the form $\Delta_{d-1}(2)$.
    By Proposition \ref{prop:rank_prop} (i), condition (i) of the Degeneracy Criterion is satisfied at the simple nodes
    (these are the $(d-2)$-simplices).
    Thus, we only need to check the Degeneracy Criterion conditions at the nodes corresponding to the vertices.
    Let $\v_i$ be a vertex of~$P$.
    The corresponding node is adjacent to exactly $d-2$ later nodes $F_{i_1}, \cdots, F_{i_{d-2}}$.
    These nodes correspond to facets of the form $\Delta_{d-1}(2)$.
    By Proposition \ref{prop:hypersimplex_structure} (iii), any $d-3$ of them are the facet set of a $2$-face.
    Thus, by Proposition \ref{prop:rank_prop} (ii),
    condition (ii) of the Degeneracy Criterion is satisfied at the nodes corresponding to the vertices.
\end{proof}

\section{Negative Results}\label{sec:4dim}
In this section, we discuss negative results of various flavors. The first set of negative results shows that there are $4$-polytopes with the property that their realization spaces are smooth manifolds but the dimension is not the expected one, i.e.~not equal to the natural guess, providing explicit smooth counterexamples to the false claim in \cite{robertson1984}. In the second part, we show that the realization space of the 24-cell is not a smooth manifold, at least in its natural embeddings.

\subsection{The Smallest Polytope $P$ with $\dim \RC{P} \not = \NG{P}$}\label{subsec:three_smallest_counter-examples}
Testing Theorem \ref{thm:almost_3degen} on the database of
all $4$-polytopes with at most $9$ vertices \cite{Firsching-9vertices} gives the following table.

\begin{table}[H] \label{tbl:almost_3degen}
    \begin{center}
    \begin{tabular}{ |c|c|c|c| }
    \hline
        $n$     &   $ |\mathcal{P}_n^4|$    & $|\mathcal{A}_n^4|$   & $|\mathcal{P}_n^4 \setminus \mathcal{A}_n^4|$   \\
        \hline
        5       & 1                         & 1                     &   0     \\
        6       & 4                         & 4                     &   0     \\
        7       & 31                        & 31                    &   0     \\
        8       & 1294                      & 1287                  &   7     \\
        9       & 274148                    & 272668                &   1480  \\
    \hline
    \end{tabular}
    \caption{$\mathcal{P}_n^4$ is the set of all $4$-polytopes with $n$ vertices.
	$\mathcal{A}_n^4$ is the subset of $\mathcal{P}_n^4$ of all polytopes which satisfy the condition in Theorem \ref{thm:almost_3degen}.}
\end{center}
\end{table}

By examining closely the seven $4$-polytopes with $8$ vertices that
do not satisfy the condition in Theorem \ref{thm:almost_3degen},
we were able to show that four of them satisfy the Degeneracy Criterion (Theorem \ref{thm:degeneracy}).
Thus, in this section, we are going to look at the remaining three $4$-polytopes with $8$ vertices:
We will call them $P_1, P_2$ and $P_3$, and we will show that they do not satisfy the Degeneracy Criterion.
We will also find the dimensions of their realizations spaces,
which turn out to be different from the respective natural guess.
Thus, we found the smallest examples (in terms of the number of vertices)
where the dimensions of their realization spaces are different from their natural guesses.
These three polytopes can be constructed from a pyramid over a triangular prism $\mathrm{pyr}(\mathrm{prism}(\Delta))$ by adding a new point as follows:
\begin{enumerate}[(1)]
    \item In $P_1$, the new point should lie beyond the triangular prism facet, on the supporting hyperplane of a simplex facet and beneath all the other facets.
        This polytope has $8$ vertices, $9$ facets and $43$ vertex-facet incidences.
        A realization of $P_1$ is given by the following vertex description.
        \begin{align*}
            \kbordermatrix{
                & \v_0 & \v_1 & \v_2 & \v_3 & \v_4 & \v_5 & \v_6 & \v_7 \\
                & -1 & -1 & -1 & 1 & 1 & 1 & 1 & 1 \\
                & 1 & -1 & 0 & 1 & -1 & 0 & 0 & 0 \\
                & 1 & 1 & -1 & 1 & 1 & -1 & 0 & 0 \\
                & 0 & 0 & 0 & 0 & 0 & 0 & -1 & 1
            }.
        \end{align*}
    \item In $P_2$, the new point should lie beyond the triangular prism facets and beneath all the other facets.
		In other words, $P_2$ is a bipyramid over the triangular prism.
		
        This polytope has $8$ vertices, $10$ facets and $46$ vertex-facet incidences.
        A realization of $P_2$ is given by the following vertex description.
        \begin{align*}
            \kbordermatrix{
                & \v_0 & \v_1 & \v_2 & \v_3 & \v_4 & \v_5 & \v_6 & \v_7 \\
                & -1 & -1 & -1 & 1 & 1 & 1 & 0 & 0 \\
                & 1 & -1 & 0 & 1 & -1 & 0 & 0 & 0 \\
                & 1 & 1 & -1 & 1 & 1 & -1 & 0 & 0 \\
                & 0 & 0 & 0 & 0 & 0 & 0 & -1 & 1
            }.
        \end{align*}

    \item In $P_3$, the new point should lie beyond the triangular prism facet, beyond a simplex facet and beneath all the other facets.
        This polytope has $8$ vertices, $11$ facets and $50$ vertex-facet incidences.
        A realization of $P_3$ is given by the following vertex description.
        \begin{align*}
            \kbordermatrix{
                & \v_0 & \v_1 & \v_2 & \v_3 & \v_4 & \v_5 & \v_6 & \v_7 \\
                & -1 & -1 & -1 & 1 & 1 & 1 & 0 & 0 \\
                & 1 & -1 & 0 & 1 & -1 & 0 & 0 & 0 \\
                & 1 & 1 & -1 & 1 & 1 & -1 & 0 & 0 \\
                & 0 & 0 & 0 & 0 & 0 & 0 & -1 & 1
            }.
        \end{align*}
\end{enumerate}

To start with, we introduce notation for a classical notion of realization spaces.
\begin{definition}\label{def:RS}
    Let $P$ be a labeled $d$-polytope with $n$ vertices.
    The \emph{realization space} of~$P$ is the set
    $$\RS{P} = \{ V \in \R^{d \times n} \mid \conv{V} \textrm{ realizes } P \}.$$
\end{definition}

\begin{remark}\label{remark:Rh_and_RC}
    Our various models of realization spaces that we considered so far fit together nicely. Let $P\subset \R^d$ be a $d$-dimensional polytope and denote by $\wh{P}\subset \R^{d+1}$ the polyhedral cone generated by $\{(1,\x)\in\R^{d+1}\vert \, \x\in P\}$. The centered realization space is diffeomorphic to an open subset of $\RS{P}$ by the map $(V,A)\mapsto V$. The set $\RS{P}$ is diffeomorphic to an open subset of $\Rh(\wh{P})$ by the process of homogenization (with appropriate choice of scaling for the ray and facet description of the resulting cone, namely scaling the columns of these matrices to have norm $1$). In particular, if we can show that $\Rh(\wh{P})$ is a smooth manifold, then the spaces $\RS{P}$ and $\RC{P}$ are smooth manifolds of dimension $\dim \Rh(\wh{P})$. 
\end{remark}

Richter-Gebert introduced the following notion in \cite{Jurgen1996} for polytopes that we also use in a cone version.
\begin{definition}
    A $d$-polytope is \emph{necessarily flat} if any polyhedral embedding of its boundary complex in $\R^n$, where $n \geq d$, has affine dimension at most $d$.
    A $(d+1)$-dimensional closed and pointed polyhedral cone is \emph{necessarily flat} if any polyhedral embedding of its boundary fan in~$\R^n$, where $n \geq d+1$, has affine dimension at most $d+1$.
\end{definition}
Here a \emph{polyhedral embedding} of a $d$-dimensional polyhedral complex $\mathcal{C}$ in $\R^n$ is a mapping of the vertices of $\mathcal{C}$ into $\R^n$ such that the image of every $k$-face of $\mathcal{C}$ is a $k$-dimensional convex polyhedron combinatorially equivalent to that $k$-face, and for no two faces the images intersect in their relative interiors. 

In dimension $2$, the only necessarily flat polytope is the triangle.
In dimension $3$, Richter-Gebert in \cite[Lemma 3.2.6]{Jurgen1996} showed that pyramids, prisms and ``tents'' over $n$-gons are necessarily flat.
Using the same proof provided by Richter-Gebert, one can show that the polyhedral cone arises from a 3-dimensional prism is necessarily flat.
Another related result is by Schwartz \cite[Lemma 2.6]{Schwartz2004} who showed that every simple $d$-polytope is necessarily flat, for $d \geq 3$

\begin{proposition}\label{prop:smallest_examples}
    The realization spaces of $P_1$, $P_2$ and $P_3$ are smooth manifolds of the following dimensions:
    \begin{align*}
        \dim \RC{P_1} = 26  >  25 = \NG{P_1}, \\
        \dim \RC{P_2} = 27  >  26 = \NG{P_2}, \\
        \dim \RC{P_3} = 27  >  26 = \NG{P_3}.
    \end{align*}
\end{proposition}

\begin{proof}
    The boundary fan of each of the cones $\wh{P_1}, \wh{P_2}$ and $\wh{P_3}$ contains the boundary fan of the cone over a triangular prism (with a missing simplex facet for $\wh{P_1}$), which we call $\mathcal{F}$. 
    The fan $\mathcal{F}$ has dimension $4$ in any realization of $\wh{P_1}, \wh{P_2}$ and $\wh{P_3}$ since the cone of a triangular prism is necessarily flat. Due to convexity, we know that the span of $\mathcal{F}$ has dimension $4$ in any realization of $\wh{P_1}, \wh{P_2}$ and $\wh{P_3}$.
    Thus, every realization of $\wh{P_i}$ is obtained from a realization of $C$, where $C$ is the cone of a pyramid over a triangular prism.
    The apex ray $\x$ can be chosen in an open subset of a linear space that varies differentiably with the realization of $C$.
    In $P_2$ and $P_3$, this open subset has dimension $5-1$ since $\x$ should have norm $1$ and the only other constraints are strict inequalities which correspond to the conditions that $\x$ lies beneath or beyond a facet.
    In $P_1$, this open subset has dimension $5-2$ since, in addition to the constraints mentioned in $P_2$ and $P_3$, $\x$ should lie on the supporting hyperplane of a simplicial facet of $C$.
    We are done knowing that $\Rh(C)$ is a smooth manifold of dimension $23$ by the Homogeneous Degeneracy Criterion \Cref{thm:degeneracyhom}, and using \Cref{remark:Rh_and_RC}.
\end{proof}

\begin{theorem}\label{thm:4-polytopes-up-to-8-vertices}
    Let $P$ be a $4$-polytope with at most $8$ vertices. Then the realization space $\RC{P}$ is a smooth manifold.
    Its dimension is equal to $\NG{P}$, except for the three polytopes $P_1, P_2$ or $P_3$,
	where the dimension of the realization space is $\NG{P_i}+1$.
    \qed
\end{theorem}

\begin{proposition}\label{prop:simplenecessarilyflat}
    For $d\geq 3$, let $Q$ be a simple $d$-polytope that is not a simplex.
    Then, the realization space of $P:=\bipyr(Q)$, the bipyramid over $Q$, is also a manifold but of dimension strictly greater than $\NG{P}$.
\end{proposition}

\begin{proof}
    Using Proposition \ref{prop:ng}(v), we have
    \[\NG{P} = 2\dim \RC{Q} - (d-1)f_0(Q) + 2(d+1).\]
    By Schwartz \cite[Lemma 2.6]{Schwartz2004}, $Q$ is necessarily flat.
    Thus, any realization of $P$ can be obtained by embedding a realization of $Q$ into a hyperplane $H$ in $\R^{d+1}$
    and then adding two points $\v$ and $\v'$, each in a different open half space of $H$,
    such that the segment $[\v, \v']$ intersects the relative interior of $Q$. In particular,
    \[\dim \RC{P} = \dim \RC{Q} + 3(d+1).\]
    Thus,
    \[\dim \RC{P} - \NG{P} = (d+1) + (d-1)f_0(Q) - \dim \RC{Q}.\]
    Since $Q$ is simple,
    \[\dim \RC{P} - \NG{P} = (d+1) + (d-1)f_0(Q) - d f_{d-1}(Q).\]
    We are done if we show that the right-hand-side is positive.
    For this we use the Lower Bound Theorem by Barnette \cite{barnette1971} \cite{Brondsted}, to derive
    \begin{align*}
        f_0(Q) &\geq (d-1)f_{d-1}(Q) - (d+1)(d-2) \\
               &= \tfrac{d}{d-1} f_{d-1}(Q) + \big((d-1) - \tfrac{d}{d-1}\big) f_{d-1}(Q) - (d+1)(d-2) \\
               &> \tfrac{d}{d-1} f_{d-1}(Q) + \big((d-1) - \tfrac{d}{d-1}\big) (d+1) - (d+1)(d-2) \\
               &= \tfrac{d}{d-1} f_{d-1}(Q) - \tfrac{d+1}{d-1},
    \end{align*}
    where the strict inequality comes from the fact that $Q$ is not a simplex.
\end{proof}

Thus Proposition~\ref{prop:simplenecessarilyflat} implies that the realization space
is a manifold of dimension greater than $\NG{P}$ both for the bipyramid over a triangular prism $P_2$ of Theorem \ref{thm:4-polytopes-up-to-8-vertices},
as well as for the bipyramids over duals of cyclic polytopes of Example~\ref{examples:bipyramid-over-dual-to-cyclic}.
	
\subsection{The 24-cell}\label{subsec:24cell}
Arguably the 24-cell $\twentyfourcell$
is one of the most interesting and unique examples in polytope theory.
It was discovered by Ludwig Schläfli around 1850, but his work was published only in 1901 \cite{Schlafli1901}.
For a classical discussion see e.g.\ Coxeter \cite{Coxeter1963}.
Its symmetry group is the Coxeter--Weyl group $F_4$ of order $1152$.
The 24-cell is unique in many ways.
For example, it is the only centrally-symmetric self-dual regular polytope.
It is also the only regular polytope that is not simple or simplicial.
Thus, in particular, by Corollary \ref{cor:simple_and_simplicial}
for every regular polytope $P$ --- except for possibly the 24-cell --- 
the realization space is a smooth manifold of dimension $\NG{P}$.

The 24-cell has $24$ vertices (and $24$ facets, whence the name) and $144$ vertex-facet incidences.
Thus the Jacobian Matrix has the format $f_{0,3}\times d(f_0+f_3)=144 \times 192$.
Full rank would mean full row rank $144$, and the natural guess is
\[
	\NG{\twentyfourcell} = 192-144 = 48.
\]
The $24$ facets of the 24-cell are octahedra (which have $6$ vertices each), 
and each vertex of the 24-cell lies in exactly six of these octahedra.
Thus its vertex-facet incidence graph is $6$-regular, so it is not $4$-degenerate,
and thus the Degeneracy Criterion cannot be applied.
Instead, we will apply the Jacobian Criterion at specific realizations.

\begin{theorem}[Dimension estimates for the realization space of the 24-cell]
The dimension of the realization space of the 24-cell satisfies the estimates
\[
	48 \leq \dim \RC{\twentyfourcell} \leq 52.
\]
\end{theorem}

\begin{proof}
    For the lower bound, we give a realization of the 24-cell where the Jacobian matrix has full rank
    (See the third item of the following proposition).
    
    For the upper bound, we use the Remark \ref{remark:degeneracy} to find a $(140 \times 192)$-submatrix of the Jacobian matrix, which has full rank at all centered realizations.
    For this, we construct an ordering of the nodes of the vertex-facet incidence graph of ${\twentyfourcell}$.
    Start with four nodes $\v, \u, F, G$ such that $\v, \u \in F \cap G$.
    For two vertices in the ordering which form an edge, add the facets which contain it and are not already in the ordering.
    For two facets in the ordering which intersect in a ridge, add the vertices which form the ridge and are not already in the ordering.
    Repeat this until all the vertices and the facets are ordered.
    This process does not get stuck because the edge-ridge graph is connected; see Sallee \cite{sallee1967}.
    If we remove the four edges between the nodes $\{\v, \u, F, G\}$, we get a $4$-degenerate ordering which satisfies the degeneracy criterion.
    A facet in this ordering was added because an edge of it appeared before.
    So for each facet, there are at most $4$ of its vertices appearing after it,
    and these four vertices are always affinely independent because they miss two vertices which form an edge of the octahedral facet.
    Similar argument for the vertices.
    We are done by Remark \ref{remark:degeneracy}.
\end{proof}

The group of projective transformations on $\R^4$ has dimension $24$,
thus this establishes that the dimension of the realization space of $\twentyfourcell$ modulo projective transformations satisfies
\[
		\dim(\RC{\twentyfourcell}/\PGL{4})\ \ge\ \ 24,
\]
where the best previous estimate, due to Paffenholz \cite{Paffenholz-diss,Paffenholz2006} was
\[
		\dim(\RC{\twentyfourcell}/\PGL{4})\ \ge\ \  4,
\]
and thus
\[
		\dim(\RC{\twentyfourcell})\ \ge\ \ 28.
\]

\begin{proposition}\label{prop:24-cell}
	The Jacobian matrix for the 24-cell $\twentyfourcell$ has different ranks at different realizations:
	\begin{enumerate}[\rm(1)]
		\item For any realization of the 24-cell as a regular polytope, the Jacobian matrix has rank $140$ (i.e., rank deficit $4$).
		\item For the non-regular Paffenholz realizations of the 24-cell,
            the Jacobian matrix has rank $\leq 142$ (i.e., rank deficit $\geq 2$).
        \item For eight $1$-parameter families of non-regular realizations
            of the 24-cell with a symmetry group of size $24$,
            the Jacobian has full rank $144$ (i.e., full rank).
	\end{enumerate}
\end{proposition}

\begin{proof}~

\begin{enumerate}[(1)]
    \item A regular realization of $\twentyfourcell$ is given by the following vertex description.
    \begin{align*}
        V_\mathrm{reg} &= \conv \left\{ \pm \vec{e}_i \pm \vec{e}_j \in \R^4 \mid 1 \leq i \leq j \leq 4 \right\}.
    \end{align*}
    At this realization, the Jacobian matrix has rank $140$, which is not full.

    \item Paffenholz \cite[Table 3.6]{Paffenholz-diss} \cite[Table 4.4]{Paffenholz2006} constructed a $4$-parameter family of realizations of the 24-cell.
    One way to see his construction is to start with the standard $4$-cube $[-1,1]^4$,
    and choose a point $(a,b,c,d) \in (-1,1)^4$ in the interior of that cube.
    Now reflect this point through the $8$ supporting hyperplanes of the facets of that cube to get $8$ new points.
    The convex hull of the $16$ vertices of the standard cube
    and the $8$ new points gives the $4$-parameter family.
    The rows of the following matrix describe the vertices of this $4$-dimensional family of realizations.

    \begin{table}[H]
        \begin{align*}
        \left[\begin{array}{rrrr}
            -1 & -1 & -1 & -1 \\
            1 & 1 & -1 & -1 \\
            1 & -1 & 1 & -1 \\
            -1 & 1 & 1 & -1 \\
            1 & -1 & -1 & 1 \\
            -1 & 1 & -1 & 1 \\
            -1 & -1 & 1 & 1 \\
            1 & 1 & 1 & 1 \\
            1 & -1 & -1 & -1 \\
            -1 & 1 & -1 & -1 \\
            -1 & -1 & 1 & -1 \\
            1 & 1 & 1 & -1 \\
            -1 & -1 & -1 & 1 \\
            1 & 1 & -1 & 1 \\
            1 & -1 & 1 & 1 \\
            -1 & 1 & 1 & 1 \\
            a & b & c & -d - 2 \\
            a & b & -c + 2 & d \\
            a & b & c & -d + 2 \\
            a & b & -c - 2 & d \\
            a & -b + 2 & c & d \\
            -a - 2 & b & c & d \\
            a & -b - 2 & c & d \\
            -a + 2 & b & c & d
        \end{array}\right].
        \end{align*}
    \end{table}

    If all the parameters are zero, we get a regular realization,
    and the Jacobian matrix of the corresponding characteristic map at this realization has rank $140$.
    Otherwise, using \textsc{SageMath} \cite{sagemath}, we found two linear dependencies between the rows of the (symbolic) Jacobian matrix.

    \item Finally, we give a new construction of eight $1$-parameter families of realizations of the 24-cell.
	
    To get these families, start with the standard $4$-cube $[-1,1]^4$.
    Let $\v_i$ denote a point beyond the facet of $C_4$ defined by
    $\{ \x \in \R^4 \mid x_i = 1 \}$ and beneath all the other facets.
    Similarly, let $\u_i$ denote a point beyond the facet of $C_4$ defined by
    $\{ \x \in \R^4 \mid x_i = -1 \}$ and beneath all the other facets.
    Our goal is to construct a 24-cell whose vertex set is
    $\{-1,1\}^4 \cup \{\v_1, \cdots, \v_4, \u_1, \cdots, \u_4\}$
    such that it has the following symmetries:
    \begin{align*}
        \u_i = -\v_i \quad& \forall 1 \leq i \leq 4. \\
        \v_i \mapsto \v_j \mapsto \u_i \mapsto \u_j \mapsto \v_i \quad& \forall 1 \leq i \not= j \leq 4. \\
        \v_i \mapsto \u_j \mapsto \u_i \mapsto \v_j \mapsto \v_i \quad& \forall 1 \leq i \not= j \leq 4.
    \end{align*}
    Letting $\v_1 = (a,b,c,d)$ and looking at the symmetries of $C_4$ which map the points as described above,
    we see that we have only two options for the coordinates of the $\v_i$'s.

    \begin{minipage}{0.5\textwidth}
        \begin{align*}
        \kbordermatrix{
                    &       &       &       &       \\
            \v_1    &   a   &   b   &   c   &   d   \\
            \v_2    &   -b  &   a   &   -d  &   c   \\
            \v_3    &   -c  &   d   &   a   &   -b  \\
            \v_4    &   -d  &   -c  &   b   &   a
        }
        \end{align*}\\
    \end{minipage}%
    \begin{minipage}{0.5\textwidth}
        \begin{align*}
        \kbordermatrix{
                    &       &       &       &       \\
            \v_1    &   a   &   b   &   c   &   d   \\
            \v_2    &   -b  &   a   &   d   &   -c  \\
            \v_3    &   -c  &  -d   &   a   &   b   \\
            \v_4    &   -d  &   c   &   -b  &   a
        }
        \end{align*}\\
    \end{minipage}

    The convex hull of
    $\{-1,1\}^4 \cup \{\v_1, \cdots, \v_4, \u_1, \cdots, \u_4\}$
    is a 24-cell if and only if each line segment between a pair $(\v,\v')$ of non-opposite points of
    $\{\v_1, \cdots, \v_4, \u_1, \cdots, \u_4\}$
    intersects the (relative) interior of the $2$-face of $C_4$ defined by $F_{\v} \cap F_{\v'}$,
    where $F_{\v}$ (resp.\ $F_{\v'}$) is the facet of $C_4$ which $\v$ (resp.\ $\v'$) lies beyond.
    Writing up these conditions we see that they are equivalent to the following conditions.
    \[
        b^2, c^2, d^2 = -a^2 + 2a, \quad
        1 < a \leq 2.
    \]
    The above equations can rationally parameterized; Putting $a =  \frac{2}{(x^2+1)}$ gives
    \[
        b,c,d = \pm \frac{2x}{(x^2+1)}, \quad
        0 \leq x < 1.
    \]
    The plus-minus signs are independent, so this corresponds to $2^3 = 8$ families.
    The following matrix describes the vertex description of these families,
    where $s_1, s_2, s_3 \in \{-1,1\}^3$ and $0 \leq x < 1$.
    
    \begin{align*}
        \left[\begin{array}{rrrr}
        -1 & -1 & -1 & -1 \\
        -1 & -1 & -1 & 1 \\
        -1 & -1 & 1 & -1 \\
        -1 & -1 & 1 & 1 \\
        -1 & 1 & -1 & -1 \\
        -1 & 1 & -1 & 1 \\
        -1 & 1 & 1 & -1 \\
        -1 & 1 & 1 & 1 \\
        1 & -1 & -1 & -1 \\
        1 & -1 & -1 & 1 \\
        1 & -1 & 1 & -1 \\
        1 & -1 & 1 & 1 \\
        1 & 1 & -1 & -1 \\
        1 & 1 & -1 & 1 \\
        1 & 1 & 1 & -1 \\
        1 & 1 & 1 & 1 \\
        -\frac{2}{x^{2} + 1} & -\frac{2 \, s_{1} x}{x^{2} + 1} & -\frac{2 \, s_{2} x}{x^{2} + 1} & -\frac{2 \, s_{3} x}{x^{2} + 1} \\
        \frac{2}{x^{2} + 1} & \frac{2 \, s_{1} x}{x^{2} + 1} & \frac{2 \, s_{2} x}{x^{2} + 1} & \frac{2 \, s_{3} x}{x^{2} + 1} \\
        \frac{2 \, s_{1} x}{x^{2} + 1} & -\frac{2}{x^{2} + 1} & \frac{2 \, s_{3} x}{x^{2} + 1} & -\frac{2 \, s_{2} x}{x^{2} + 1} \\
        -\frac{2 \, s_{1} x}{x^{2} + 1} & \frac{2}{x^{2} + 1} & -\frac{2 \, s_{3} x}{x^{2} + 1} & \frac{2 \, s_{2} x}{x^{2} + 1} \\
        \frac{2 \, s_{2} x}{x^{2} + 1} & -\frac{2 \, s_{3} x}{x^{2} + 1} & -\frac{2}{x^{2} + 1} & \frac{2 \, s_{1} x}{x^{2} + 1} \\
        -\frac{2 \, s_{2} x}{x^{2} + 1} & \frac{2 \, s_{3} x}{x^{2} + 1} & \frac{2}{x^{2} + 1} & -\frac{2 \, s_{1} x}{x^{2} + 1} \\
        \frac{2 \, s_{3} x}{x^{2} + 1} & \frac{2 \, s_{2} x}{x^{2} + 1} & -\frac{2 \, s_{1} x}{x^{2} + 1} & -\frac{2}{x^{2} + 1} \\
        -\frac{2 \, s_{3} x}{x^{2} + 1} & -\frac{2 \, s_{2} x}{x^{2} + 1} & \frac{2 \, s_{1} x}{x^{2} + 1} & \frac{2}{x^{2} + 1}
        \end{array}\right]
        \end{align*}

    If $x=0$, we get a regular realization,
    and the Jacobian matrix of the corresponding characteristic map at this realization has rank $140$.
    Otherwise, when $0 < x < 1$, we get a non-regular realization.
    Using \textsc{SageMath}, we were able to show that the Jacobian matrix at the open interval
    $0 < x < 1$ has full rank.
\end{enumerate}
\end{proof}

\begin{corollary}\label{cor:smoothmanifold24}
	There is an open subset of the realization space $\RC{\twentyfourcell}$
	that is a smooth manifold of dimension $48$.
\end{corollary}

\begin{proof}
	This holds in the neighborhood of the new realizations of Proposition~\ref{prop:24-cell}(3),
	where the Jacobian has full rank.
\end{proof}

However, in contrast to the local situation announced by Corollary~\ref{cor:smoothmanifold24},
there are other points (e.g.\ at the regular realization) where the Jacobian property fails.
The following theorem shows that the $\RC{\twentyfourcell}$ is not smooth at any point that
corresponds to a realization that is (projectively-equivalent to) a regular polytope.

\begin{theorem}\label{thm:c24notsmooth}
	The realization space $\RC{\twentyfourcell}\subset\R^{192}$ is not smooth at 
	any point that corresponds to a realization of the 24-cell as a regular polytope.
\end{theorem}

\begin{proof}
    For each of the $8$ families we constructed above, we construct the corresponding Jacobian matrix $J_{s_1,s_2,s_3}$ as a function of $x$.
    These Jacobian matrices have full rank in the interval $0 < x < 1$. Thus, their kernels define ($48$-dimensional) tangent spaces along these families.
    We computed the limits of these $8$ (symbolic) tangent spaces at $x=0$, and we 
	obtained that these limits give $4$ different $48$-dimensional subspaces.
    Thus the regular realization, given by $x=0$, is not a smooth point in $\RC{\twentyfourcell}$.

    These computations were done using \textsc{SageMath} \cite{sagemath}.
    To be able to compute the limits, we did the following.
    We computed the bases for the kernels in an echelon form and then we orthogonalized (but not normalized!) the rows of these bases using Gram-Schmidt process. The entries of this echelon form are now rational functions in $x$. This produced rows in the Jacobian matrix with entries that go to infinity as $x$ goes to $0$. Those rows, we multiplied by $x$,
	 after which the entries all became convergent.
\end{proof}

\begin{remark}
	Bates, Hauenstein, Peterson \& Sommese \cite{bates2009} and Wampler, Hauenstein \& Sommese \cite{wampler2011} 
	introduced a local dimension numerical test based on the growth rate of the corank of the Macaulay matrix of the given variety after adding to it some number of random linear equations passing through the point at which we want to compute the local dimension. We used this test and we got that
    \begin{itemize}
        \item the local dimension of the realization space of the 24-cell at the regular realization is at most $50$, and
        \item the local dimension of the realization space of the 24-cell at a Paffenholz realization is at most $49$.
    \end{itemize}
\end{remark}

\section{Comparison to Other Models}
In what follows, we discuss various other models for realization spaces of polytopes and how they compare to the centered realization space. In particular, we argue that the results from \Cref{subsec:24cell} translate to these other models.

The most naive model perhaps is to record the vertex coordinates of the realization, which leads to the set $\{V\in \R^{d\times n}\colon \conv(V) \text{ realizes } P\}$, see Definition \ref{def:RS}. 
This was, for example, used in Mn\"ev's original statement of the Universality Theorem for polytopes \cite{Mnev1988}.
The centered realization space is diffeomorphic to an open subset of the naive model, and the action of the affine group acts transitively on the naive model. So this naive approach is essentially the same as the centered realization space model. One might want to factor out the action of the affine group or related transformation group actions. Several different ways to do this have been proposed in the literature and they lead to slightly different models for the realization space of a polytope, see 
Gouveia, Macchia \& Wiebe~\cite{gouveia2020combining}.

\subsection{Realization spaces modulo transformation group actions}

We begin with the model favored by Richter-Gebert in his work on the universality theorem for $4$-polytopes \cite{Jurgen1996,Z45}. Here, we factor out affine transformations by fixing points, which by the combinatorial structure of the polytope have to be affinely independent in every realization, to be the origin and the standard basis vectors. The following proposition tells us how we can identify this model explicitly with a subset of the centered realization space (of lower dimension).

\begin{proposition}\label{prop:rgaffinetrans}
    Let $P$ be a $d$-dimensional polytope. Let $\mathcal{RG}(P)$ be the realization space of~$P$ in Richter-Gebert's model, which fixes vertices $\v_0,\v_1,\ldots,\v_d$ of~$P$ to be $\vec{e}_0 = \vec{0}$, $\vec{e}_1$, \ldots, $\vec{e}_d$.
    Let $\x_0,\x_1,\ldots,\x_d$ be affinely independent vectors in $\R^d$. Richter-Gebert's model is diffeomorphic to the space of realizations of~$P$ with $\v_i = \x_i$. 
\end{proposition}

\begin{proof}
    Let $A$ be the linear map that maps $\vec{e}_i$ to $\x_i-\x_0$. This linear transformation has full rank. Therefore, the affine transformation $\x\mapsto A\x+\x_0$ that maps $\vec{e}_i$ to $\x_i$ is invertible. This map induces a diffeomorphism of the described realization spaces.
\end{proof}

A recent way to encode realizations of a polytope is the \emph{slack realization space} introduced 
by {Gouveia}, {Macchia}, {Thomas} \& {Wiebe} \cite{zbMATH07105541}. Essentially, the authors show that realizations of a polytope correspond to $n\times m$ matrices of rank $d+1$ (whose rows are indexed by vertices and columns by facets) with nonnegative entries,
where zero entries appear only at the positions that correspond to vertex-facet incidences. In this model it is particularly easy to interpret and analyze the quotients modulo transformation groups. The connection of the slack model to Richter-Gebert's model discussed above as well as to the point of view of chirotopes (or oriented matroids with the same face lattice as the polytope) are explored in Gouveia, Macchia \& Wiebe~\cite{gouveia2020combining}. All different models modulo transformation groups are at least \emph{birational}, which is to say isomorphic on an open subset (as subsets of algebraic varieties, so in particular also locally diffeomorphic wherever the map is defined).

For instance, if $P\subset \R^d$ be a $d$-dimensional polytope whose first $d+1$ vertices are affinely independent in any realization of $P$, Richter-Gebert's model fixes the first $d+1$ vertices of the polytope to be $\vec{0}$, $\vec{e}_1$, \ldots, $\vec{e}_d$. In the Grassmannian model of the realization space, a realization $V$ of $P = \conv(V)$ is mapped to the column space of $\wh{V}$. By the choice of the first $d+1$ columns of $V$, the Plücker vector of such a realization always ends up in the same canonical affine chart of the Grassmannian given by the first $(d+1)\times(d+1)$ block of $\wh{V}$ having full rank (so that the corresponding entry in the Plücker vector is non-zero, more precisely $1$). So Richter-Gebert's model is naturally a subset of the Grassmannian model of realizations of $\wh{P}$. The slack realization space is birational to the Grassmannian model by \cite[Theorem~4.7]{gouveia2020combining}. 

We discuss how to translate smoothness results from the centered realization space model to these quotient models exemplarily for the $24$-cell in the following section.

\subsection{The 24-cell in other models for the realization space}

\begin{proposition}\label{prop:24cellRG}
    There is an open neighborhood of the regular realization of the 24-cell in $\mathcal{RG}(\twentyfourcell)$ that is diffeomorphic to a transversal affine section of $\RC{\twentyfourcell}$. In particular, the regular realization is also a singular point in $\mathcal{RG}(\twentyfourcell)$.
\end{proposition}

\begin{proof}
    Choose the vectors $\x_0 = (-1,1,1,1)$, $\x_1 = (1,1,1,1)$, $\x_2 = (0,2,0,0)$, $\x_3 = (0,0,2,0)$, $\x_4 = (0,0,0,2)$. These are affinely independent so that \Cref{prop:rgaffinetrans} implies that $\mathcal{RG}(\twentyfourcell)$ is diffeomorphic to all realizations of the 24-cell such that $5$ vertices have the above coordinates, which we call $\mathcal{RG}'(\twentyfourcell)$. These are chosen as they are vertices of a regular realization. 
There is an open neighborhood of this regular realization in $\mathcal{RG}'(\twentyfourcell)$ that lies in $\RC{\twentyfourcell}$. This neighborhood is an affine section $\RC{\twentyfourcell}\cap L$ of the centered realization space determined by the affine conditions that the $5$ vertices $\v_7, \v_{15}, \v_{19}, \v_{21}, \v_{23}$ are equal to $\x_0, \cdots, \x_4$ respectively. We can now show that claim by a computation. We consider the same four curves as in \Cref{thm:c24notsmooth} in $\RC{\twentyfourcell}$ that approach the regular realization. We transform them into curves in $\mathcal{RG}'(\twentyfourcell)$ by the affine transformation moving the five chosen vertices to the fixed $x_i$. Sufficiently close to the regular realization (given by the parameter value $0$), these transformed curves in $\mathcal{RG}'(\twentyfourcell)$ pass through smooth points in $\mathcal{RG}'(\twentyfourcell)$ by generic smoothness of the quotient map. The tangent space to $\mathcal{RG}'(\twentyfourcell)$ is a $28$-dimensional linear space depending on $m$. The limit for $m=0$ can be computed by intersecting the limits of the tangent spaces of the original curves inside $\RC{\twentyfourcell}$ with the linear space ${\rm lin}(L)$ corresponding to the affine subspace $L$, which we compute to be a $28$-dimensional subspace. In fact, as in the proof of \Cref{thm:c24notsmooth}, we obtain four different $28$-dimensional subspaces, which shows that $\mathcal{RG}'(\twentyfourcell)$, and therefore $\mathcal{RG}(\twentyfourcell)$, is not smooth at the regular realization.
\end{proof}

The fact that the realization space of the 24-cell is not a smooth manifold in Richter-Gebert's model locally around a regular realization also shows that it is not a smooth manifold in the Grassmannian model and therefore neither in the slack model (see~\cite{gouveia2020combining}) because the transition maps between the models are defined locally around the regular realizations.

\end{document}